\documentclass[11pt,reqno]{amsart}
\usepackage{a4wide}
\usepackage{amsmath,amssymb,amsfonts,amsthm,dutchcal}
\usepackage{tikz}
\usetikzlibrary{decorations.pathreplacing}
\usepackage{color}
\usepackage{xcolor}
\usepackage{graphicx}
\usepackage{tikz}
\usetikzlibrary{arrows}

\usepackage{cancel}

\usepackage[colorlinks=true,urlcolor=blue,
citecolor=red,linkcolor=blue,linktocpage,pdfpagelabels,
bookmarksnumbered,bookmarksopen]{hyperref}

\theoremstyle{plain}

\newtheorem{theorem}{Theorem}[section]
\newtheorem{lemma}[theorem]{Lemma}
\newtheorem{proposition}[theorem]{Proposition}

\newtheorem{remark}[theorem]{Remark}
\newtheorem{definition}[theorem]{Definition}
\theoremstyle{definition}
\numberwithin{equation}{section}
\newcommand{\LeftShift}{-4pt}

\def\Om{\Omega}
\def\R{\mathbb{R}}
\def\S{\mathbb{S}}

\def\N{\mathbb{N}}
\def\sd{{\rm sd}}

\def\H{\mathcal{H}}

\def\cN{\mathcal{N}}

\def\C{\mathcal{C}}
\def\Co{\mathbf{C}}
\def\J{J_{\lambda,\Co}}
\def\Chi#1{\hbox{{\large $\chi$}{\Large $_{_{#1}}$}}}

\newcommand{\res}{\mathop{\hbox{\vrule height 7pt width .5pt depth 0pt
\vrule height .5pt width 6pt depth 0pt}}\nolimits}

\newcommand{\wstar}{\overset{\!\!*}\rightharpoonup}

\newcommand{\pa}{\partial}

\newcommand{\medint}{-\kern -,375cm\int}
\newcommand{\medintinrigo}{-\kern -,315cm\int}
\newcommand{\wto}{\rightharpoonup}
\newcommand{\eps}{\varepsilon}
\newcommand{\Ha}{{\mathcal{H}}}

\def\beq{\begin{equation}}
\def\eeq{\end{equation}}

\pgfarrowsdeclare{<<<}{>>>}
{
\arrowsize=0.2pt
\advance\arrowsize by .5\pgflinewidth
\pgfarrowsleftextend{-4\arrowsize-.5\pgflinewidth}
\pgfarrowsrightextend{.5\pgflinewidth}
}
{
\arrowsize=0.2pt
\advance\arrowsize by .5\pgflinewidth
\pgfpathmoveto{\pgfpointorigin}
\pgfpathlineto{\pgfpoint{-1.2mm}{-.8mm}}
\pgfusepathqstroke
\pgfpathmoveto{\pgfpointorigin}
\pgfpathlineto{\pgfpoint{-1.2mm}{.8mm}}
\pgfusepathqstroke
}%%%% DEFINIZIONE FRECCIA CARINA:FINE (per usarla, iniziare una figura con \begin{tikzpicture}[>=>>>] )

\pgfarrowsdeclare{<<<<}{>>>>}
{
\arrowsize=0.2pt
\advance\arrowsize by .5\pgflinewidth
\pgfarrowsleftextend{-4\arrowsize-.5\pgflinewidth}
\pgfarrowsrightextend{.5\pgflinewidth}
}
{
\arrowsize=0.2pt
\advance\arrowsize by .5\pgflinewidth
\pgfpathmoveto{\pgfpointorigin}
\pgfpathlineto{\pgfpoint{-1.8mm}{-1.2mm}}
\pgfusepathqstroke
\pgfpathmoveto{\pgfpointorigin}
\pgfpathlineto{\pgfpoint{-1.8mm}{1.2mm}}
\pgfusepathqstroke
}%

\title[Capillary energy outside convex sets]{The isoperimetric inequality for the capillary energy outside convex sets}

\author{Nicola Fusco}
\address{Dipartimento di Matematica e Applicazioni, Universit\`a di Napoli Federico II, Italy}\email{n.fusco@unina.it}

\author{Vesa Julin}
\address{Matematiikan ja Tilastotieteen Laitos, Jyv\"askyl\"an Yliopisto, Finland}\email{vesa.julin@jyu.fi}

\author{Massimiliano Morini}
\address{Dipartimento di Scienze Matematiche Fisiche e Informatiche, Universit\`a di Parma, Italy}\email{massimiliano.morini@unipr.it}

\author{Aldo Pratelli}
\address{Dipartimento di Matematica, Universit\`a di Pisa, Italy}\email{aldo.pratelli@unipi.it}

\begin{document}

\begin{abstract}
We study the isoperimetric problem for capillary hypersurfaces with a general contact angle $\theta \in (0, \pi)$, outside arbitrary convex sets. We prove that the capillary energy of any surface supported on any such convex set is larger than that of a spherical cap with the same volume and the same contact angle on a flat support, and we characterize the equality cases. This provides a complete solution to the isoperimetric problem for capillary surfaces outside convex sets at arbitrary contact angles, generalizing the well-known Choe-Ghomi-Ritor\'e inequality, which corresponds to the case $\theta=\frac\pi2$.
\end{abstract}

\maketitle
%\tableofcontents

\section{Introduction}

Let $\Co \subset \R^N$ be a closed convex set with nonempty interior. Given a set of finite perimeter $E\subset\R^N\setminus\Co$ and $\lambda\in(-1,1)$ we define the capillary energy as 
\[
\J(E):=P(E; \R^N \setminus \Co) - \lambda \H^{N-1}(\pa^*E \cap \Co).
\]
Here, for any Borel set $G$, $P(E;G)=\H^{N-1}(\pa E^*\cap G)$ and $\pa^*E$ is the reduced boundary of $E$ (for the definitions and the relevant properties see~\cite{AmbrosioFuscoPallara00, Maggi12}).
The capillary energy has a natural physical motivation as it models a liquid drop supported on a given substrate and we refer to~\cite{Finnbook} for a comprehensive introduction to the topic.

For every $m >0$ we consider the isoperimetric problem 
\begin{equation} 
\label{def:min-prob}
I_{\Co}(m) := \inf\{ \J(E):\,\,E \subset \R^N \setminus \Co, \, |E| =m \}\,.
\end{equation}
When the convex set $\Co$ is bounded the problem~\eqref{def:min-prob} has a minimizer, and if $\Co$ is in addition smooth, then the minimizer is smooth up to a small singular set and the free boundary $\pa E \setminus \Co$ meets the surface $\pa \Co$ with a contact angle $\theta=\arccos\lambda$ given by the classical Young's law ~\cite{Taylor77, De-PhilippisMaggi15}. We also mention the recent work related to Allard type regularity for critical sets ~\cite{gasparetto}. When the convex set is unbounded the problem~\eqref{def:min-prob} might not admit a minimizer. This happens for instance when $\Co=\C\times \R\subset \R^3$, with $\C$ the epigraph of a parabola. In this case, as a consequence of our main Theorem~\ref{thm1}, minimizing sequences slide upwards to infinity along the boundary of $\Co$ and the isoperimetric profile~\eqref{def:min-prob} agrees with the profile given by the half-space. In the case $\lambda = 0$ the problem for unbounded general convex sets $\Co$ is studied in~\cite{FMMN}.

The issue we want to address here is to find the convex sets $\Co$ for which the value of~\eqref{def:min-prob} is the smallest. 
In the case $\lambda= 0$ the problem reduces to the relative isoperimetric inequality outside convex sets proven by Choe-Ghomi-Ritor\'e in~\cite{ChGhoRi07}: using the tools developed in ~\cite{ChGhoRi06} they show that the half-space gives the lowest value for~\eqref{def:min-prob} . On the other hand, rather surprisingly the capillary case with general $\lambda \neq 0$ has remained open until now as all the methods devised for the relative isoperimetric problem do not seem to be applicable to ~\eqref{def:min-prob}. In this paper, we solve the problem completely for general $\lambda\in (-1,1)$ by adopting a different approach.

In order to state our main result we denote the half-space by $\textbf{H} = \{ x \in \R^N : x\cdot e_N < \lambda \}$ and by $B^\lambda$ the (solid) spherical cap 
\[
B^\lambda= \{ x \in B_1 : x\cdot e_N > \lambda\}\,.
\]
Notice that at the points of $B^\lambda\cap \partial\textbf H$ the angle between the outer normal to $B^\lambda$ and $e_N$ equals $\theta_\lambda=\arccos\lambda$. More in general, given any $m>0$, we let $B^\lambda[m]$ denote the spherical cap contained in $\R^N\setminus\textbf H$ having volume $m$ and such that the angle between the outer normal and $e_N$ at points of $B^\lambda\cap \partial\textbf H$ is $\theta_\lambda$. Our main result is the following. 
%
%\smallskip
%\begin{center}
%\emph{ {\small If $E \subset \R^3$ is a set of finite perimeter outside a convex cylinder $\Co$, then the capillarity energy $\J(E)$ is larger than the one given by the spherical cap on top of a halfspace. Morever, in the case of equality, the set $E$ is a spherical cap on top of a facet of the convex set. }}
%\end{center}
%\smallskip

%Denote $\Sigma := \pa^*E \setminus \Co $ and set $\theta=\arccos\lambda$. 
% We denote the spherical cap in direction $\nu \in \mathbb{S}^{N-1}$ by 
%\[
%S_{\theta, \nu} = \{ y \in \mathbb{S}^{N-1} : y_N = y \cdot e \geq \cos \theta \}.
%\]
%When $e=e_N$ we will simply write $S_{\theta}$ in place of $S_{\theta, e_N}$. Moreover for every $m>0$ we denote by $S_{\theta}[m]$ the homothetic of $S_{\theta}$ such that $|S_{\theta}[m]|=m$. 

\begin{figure}[h!]
%\vspace{-2cm}
\includegraphics[scale=0.25]{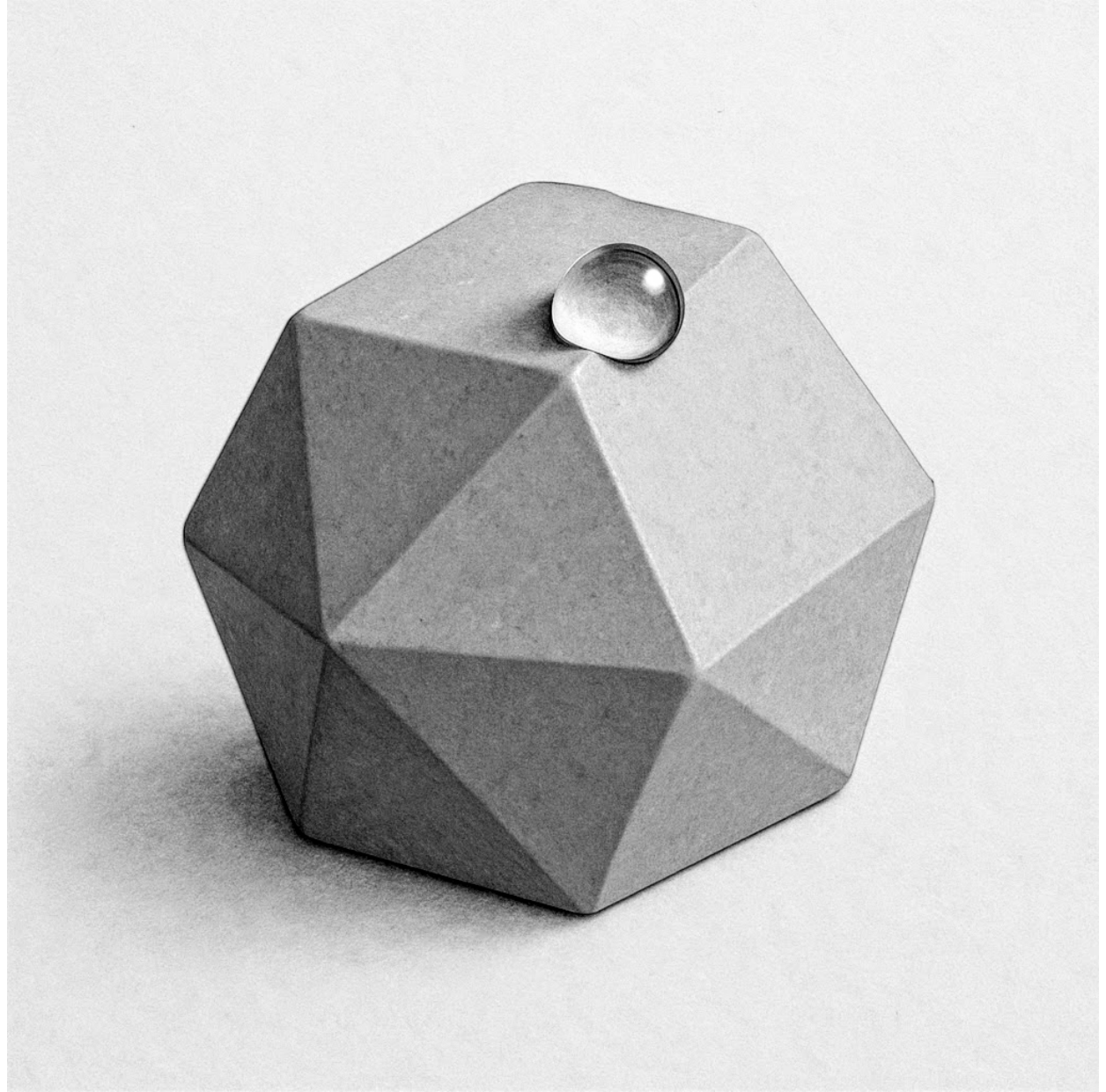}
%\put(-75,120){\Large ${W^\phi}$}
%\put(-212, 100){{\Large $B^{\phi}$}}
%\vspace{-1cm}
\caption{Droplet supported on a convex polyhedron}
\label{fig-cylinder}
\end{figure}

\begin{theorem}
\label{thm1}
Let $\lambda \in (-1,1)$ and let $\Co$ be a closed convex set. 
For every set of finite perimeter $E\subset\R^N\setminus\Co$ such that $|E|=m$ we have
\beq\label{main1}
\J(E)\geq J_{\lambda,\textbf{H}}(B^\lambda[m])\,.
\eeq
Moreover the equality holds if and only if $E$ is isometric to $B^\lambda[m]$ and $E$ sits on a flat part of the boundary of $\Co$.

\end{theorem}

In the case where $\Co$ has empty interior the capillary energy must be defined as in~\eqref{empty1}.

We highlight also that we do not assume any regularity on $\Co$. In particular, the theorem above applies to the case where $\Co$ is an infinite wedge and shows that the capillary energy of a droplet sitting outside a wedge and wetting its ridge has energy strictly larger than a spherical cap lying on a flat surface, a fact that, to the best of our knowledge, was not proven before.

Instead, the capillary isoperimetric problem \emph{inside} a convex wedge for $\lambda=0$ was studied in~\cite{lopez} where it is proved that the minimizer of the capillary energy is a portion of a ball centered at the ridge of the wedge. The same result holds also for critical points, as a consequence of the generalized Heinze-Karcher inequality proven in~\cite{JWXZ}. Instead, Theorem~\ref{thm1} implies that the capillary isoperimetric problem {\em outside} a convex wedge has the opposite behavior, as explained before.

As we already mentioned, the case $\lambda =0$ of Theorem~\ref{thm1} is the relative isoperimetric inequality due to Choe-Ghomi-Ritor\'e~\cite{ChGhoRi07}, see also ~\cite{FM2023} for the rigidity, i.e., the characterization of the equality case for general, possibly nonsmooth convex sets. We also refer to
\cite{LiuWangWeng} for an alternative proof of the same inequality and
%to~\cite{paspoz} for a quantitative version and to~\cite{jiazhang} for the quantitative version of the critical sets in the half-space and
to~\cite{krummel} for the problem in higher codimension.

In order to prove Theorem~\ref{thm1} we need to introduce some novel ideas, that will be explained in more detail in Section~\ref{sec:overview} below. Indeed, the approach based on normal cones developed in ~\cite{ChGhoRi06, ChGhoRi07} for the case $\lambda =0$ (and further refined in ~\cite{FM2023}) gives only information on the free boundary $\pa E \setminus \Co$, while the contact region $\pa E \cap \Co$ remains invisible. In order to overcome this, we develop the ABP-method for the capillary isoperimetric problem outside convex sets. The ABP approach was originally introduced by Cabr\'e for the standard isoperimetric inequality and extended to the relative isoperimetric problem outside convex sets (the case $\lambda = 0$) in ~\cite{LiuWangWeng}. Here we tackle the general case $\lambda\not=0$. As we will explain in the next subsection, the core of ABP-method relies on a subtle estimate on the measure of the set of subdifferentials of the function $u$ solving problem~\eqref{eq:neumann-0} below. Such an estimate turns out to be significantly difficult to obtain and requires new insights. 

We next give an overview of the ABP-argument which we use in the proof of Theorem~\ref{thm1}.

\subsection{Overview of the proof}\label{sec:overview}

The proof of Theorem~\ref{thm1} is based on the ABP-method applied to the Neumann problem~\eqref{eq:neumann-0}. We note that in the context of isoperimetric problems this method was used first time by Cabr\'e in ~\cite{cabre2000, cabre2008} and further generalized in~\cite{CRS2016, CGPRS}. A variant of this method has been used recently in~\cite{paspoz} to prove the stability of the isoperimetric inequality for the capillary energy in a half space (see also \cite{kreutzschmidt} and the recent preprint~\cite{carpaspoz} for a stronger quantitative version of the same isoperimetric inequality). 

In this paper we develop the ABP-method for the capillary isoperimetric problem outside convex sets. Let us sketch the proof and outline the main challenges of the argument. 

By scaling we may reduce to the case $|E| = |B^\lambda|$. Assume for simplicity that the set $E$ is regular in which case we denote it by $E=\Omega$. To be more precise we assume that $\Omega \subset \R^N\setminus\Co$ is a bounded Lipschitz domain with $|\Omega| = |B^\lambda|$, such that $\Sigma = \pa \Omega \setminus\Co$ and $\Gamma = \pa \Omega \cap \Co$ are smooth embedded manifolds with a common boundary denoted by $\gamma$. Let $u :\overline\Omega \to \R$ be the solution of the Neumann boundary problem 
\begin{equation}\label{eq:neumann-0} 
\begin{cases}
& \Delta u = c \quad \text{in } \, \Omega \\
&\pa_\nu u = 1 \quad \text{on } \, \Sigma\\ 
&\pa_\nu u = -\lambda \quad \text{on } \, \Gamma\,, 
\end{cases}
\end{equation}
where $\lambda\in(-1,1)$ and the constant 
\begin{equation}\label{eq:compatibility} 
c = \frac{\H^{N-1}(\Sigma) - \lambda \H^{N-1}(\Gamma)}{|\Omega|} = \frac{\J(\Omega )}{|\Omega|}
\end{equation}
is the one prescribed by the compatibility condition. We point out that in the case $\Omega = B^\lambda$ up to an additive constant $u(x) = \frac12 |x|^2$ and $c = N$. 

Let us now consider the points $x \in \Omega$ where $|\nabla u(x)|<1$ and such that the tangent hyperplane to the graph of $u$ at $x$ is also a supporting hyperplane. More precisely, we set
\beq\label{omegahat}
\widehat\Omega:=\big\{x\in \Om:\, |\nabla u(x)|<1\text{ and } u(y) -u(x) \geq \nabla u(x) \cdot (y-x) \quad \text{for all } \, y \in \Omega\big\}. 
\eeq
Clearly it holds $\nabla^2 u(x) \geq 0$ for all $x\in \widehat\Om$. 
% and the union of subdifferentials by $\mathcal A_u$ 
%\[
%\mathcal A_u = \bigcup_{x \in \Omega}J_{\overline \Omega}u(x). 
%\]
In order to carry on the ABP-argument one needs to show the following crucial estimate 
\begin{equation}\label{eq:crucial-est} 
|\nabla u(\widehat\Om)| \geq |B^\lambda| 
\end{equation}
by somehow exploiting the Neumann boundary conditions in~\eqref{eq:neumann-0} and the geometry of $\Co$.

Once~\eqref{eq:crucial-est} is proven, the claim then follows by using the Area Formula, the arithmetic-geometric mean inequality and recalling the value of $c$ in~\eqref{eq:compatibility} 
\[
\begin{split}
|\Omega|=|B^\lambda|\leq |\nabla u(\widehat\Om)| \leq \int_{\widehat\Om}\det\nabla^2u \, dx \leq \int_{\widehat\Om}\frac{(\Delta u)^N}{N^N}\, dx \leq \left( \frac{\J(\Omega )}{|\Omega| N}\right)^N |\Omega|\, . 
\end{split}
\]
The above chain of inequalities gives the conclusion as 
\[
J_{\lambda,\textbf{H}}(B^\lambda) = N |B^\lambda| = N|\Omega|\,. 
\]
The case of a general set of finite perimeter $E$ follows by an approximation argument. In fact, a more refined approximation argument allows us also to characterize the equality case, see Lemma~\ref{approbis}.

It is then clear that the most relevant estimate is~\eqref{eq:crucial-est} which turns out to be challenging to prove. Indeed, the inclusion $B^\lambda \subset \nabla u(\widehat\Om)$ does not hold in general and we need to develop a more subtle argument to overcome the problem. The same estimate was already proven in~\cite{LiuWangWeng} in the case $\lambda = 0$ by reformulating the problem in terms of suitable restricted normal cones to the graph of $u$ in the spirit of ~\cite{ChGhoRi06}. However their argument does not generalize to the case $\lambda \neq 0$. 

In order to deal with the case of general $\lambda$'s, we proceed as follows. We first show that the (variational) solution to~\eqref{eq:neumann-0} is a viscosity supersolution to the same problem, in the sense of Definition~\ref{def:visco}. The latter definition is stronger than the one given in ~\cite{dupuisishii} and therefore the supersolution property in the above sense does not follow from known results. Using this property we are able to relate the subdifferentials of $u$ in $\Omega$ with the subdifferentials of the restriction of $u$ to $\Gamma$ proving the following inclusion, see Lemma~\ref{lem:u-gamma},
\beq\label{crinclusion}
\mathcal B_{u_{\Gamma}}^\lambda \subset \nabla u(\widehat\Om)\, ,
\eeq
where $\mathcal B_{u_{\Gamma}}^\lambda=\bigcup_{x\in \Gamma} \{ \xi \in J_{\Gamma} u(x) : \,|\xi|<1\text{ and }\xi\cdot\nu_\Co(x)> \lambda \}$, with $\nu_\Co(x)$ the outer normal to $\Co$ at $x$ and $ J_{\Gamma} u(x)$ the subdifferential at $x$ of the restriction of $u$ to $\Gamma$; that is, 
\[
J_\Gamma u(x):=\big\{\xi\in\R^N:\,u(y) - u(x) \geq \xi \cdot (y-x) \,\,\, \text{for all} \, y \in \Gamma\big\}\,.
\]
The inclusion~\eqref{crinclusion} leads to the proof of~\eqref{eq:crucial-est} provided we are able to show that
\beq\label{introcr}
\qquad\qquad\qquad\qquad |\mathcal B_{v}^\lambda|\geq|B^\lambda| \qquad \text{for any $v:K\to\R$, with $K\subset\pa\Co$ compact.}
\eeq
In fact, it is enough to prove~\eqref{introcr} only for discrete sets $K\subset\pa\Co$, see Lemma~\ref{lem:ABP}. Note that property~\eqref{introcr} has nothing to do with Neumann problem~\eqref{eq:neumann-0} and it only depends on the geometry of $\Co$. By a simple argument based on symmetry one can easily see that closed convex sets satisfy~\eqref{introcr} for $\lambda=0$, see Section~\ref{sec:maingeo}. However, the case $\lambda\not=0$ is considerably more difficult. 

In fact, due to the aformentioned discretization procedure, proving~\eqref{introcr} is equivalent to showing that, given any convex partition $A_1,\dots,A_k$ of $\R^N$, with $A_i$ the subdifferential at $x_i$ of a function $v:\{x_1,\dots,x_k\}\subset\pa\Co\to\R$, then
\beq\label{introcr1}
\sum_{i=1}^k\big|A_i\cap\{\xi\in B_1:\,\,\xi\cdot\nu_i>\lambda\}\big|\geq|B^\lambda|\, ,
\eeq
where $\nu_i$ is the exterior normal to $\Co$ at $x_i$.

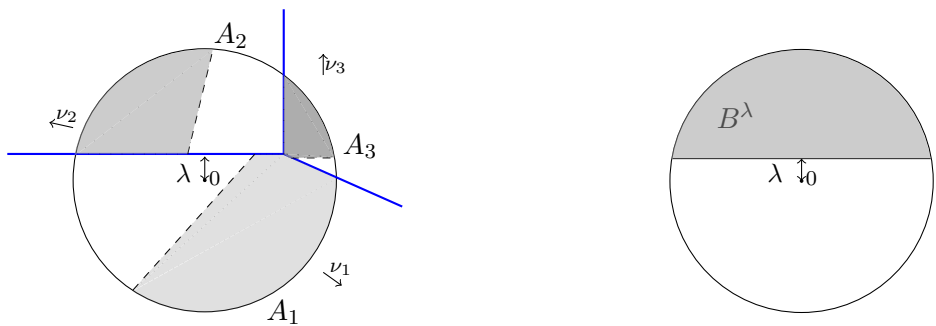
\begin{figure}[htbp]
\noindent
\begin{minipage}[c]{0.55\textwidth}
\centering
\begin{tikzpicture}[baseline={(current bounding box.center)}]
\begin{scope}[yscale=.35,xscale=.35, transform canvas={yshift=\LeftShift}]
\clip (-8, -7) rectangle (8, 7);

\draw (0,0) circle (5);
\filldraw[fill=black] (0,0) circle (.05);
\draw (0.4,-0.12) node {${}^0$}; 
\draw [thick,blue](-7.5,1) -- (3,1);
\draw [thick,blue](3,1) -- (3,6.5); 
\draw [thick,blue](3,1) -- (7.5,-1); 
\draw (5.9,1.3) node {$A_3$}; 
\draw (1,5.5) node {$A_2$}; 
\draw (3,-5) node {$A_1$}; 

\draw (5.2,-3.5) node {${}^{\nu_1}$};
\draw[->] (4.5,-3.5) -- (5.2,-4); 

%\draw[->] (0, -6) -- (0,6); 

\draw[->] (4.5, 4) -- (4.5,4.8); 
\draw (5 ,4.15) node {${}^{\nu_3}$};

\draw[->] (-5, 2) -- (-5.9,2.2);
\draw (-5.2,2.3) node {${}^{\nu_2}$}; 

\draw[dashed] (-0.65,0.99) -- (0.31,4.97); 
\draw[dashed] (1.9,0.99) -- (-2.75,-4.2); 
\draw[dashed] (3.27,0.85) -- (4.95,0.85); 
\draw[<->] (0,0) -- (0,0.9); 
\draw (-0.8,0.3) node {\small $\lambda$};

\draw[loosely dotted, fill=white!60!black,opacity=0.8, very thin] (3,1)-- (3,4) -- (4.9,0.88) -- cycle; 
\draw[loosely dotted, fill=white!60!black,opacity=0.8, very thin] (3,1)--(3.2,0.85) -- (4.95,0.88) -- cycle; 
\path[loosely dotted,fill=white!60!black,opacity=0.8, very thin] 
(4.92,0.85) arc[start angle = 10, end angle =52.8, radius = 5] -- cycle;

\draw[loosely dotted,fill=white!60!black,opacity=0.5, very thin] (-4.84,1) -- (-0.65,0.99) -- (0.3,4.95) -- cycle;
\path[loosely dotted, fill=white!60!black,opacity=0.5, very thin] 
(0.3,4.95) arc[start angle = 86, end angle =169, radius = 5] -- cycle;

\draw[dotted, fill=white!40!black,opacity=0.2, very thin] (-2.75,-4.2) -- (3,1)--(1.9,0.99) -- cycle;
\draw[dotted, fill=white!40!black,opacity=0.2, very thin] (-2.74,-4.2) -- (5,0.07) -- (3,1) -- cycle;
\path[dotted, fill=white!40!black,opacity=0.2, very thin] 
(-2.74,-4.22) arc[start angle =237, end angle =361, radius = 5] -- cycle;
\end{scope}
\end{tikzpicture}

\vspace{0.3cm}
\end{minipage}
\hfill
\begin{minipage}[c]{0.43\textwidth}
\centering
\begin{tikzpicture}[baseline={(current bounding box.center)}]
\begin{scope}[yscale=.35,xscale=.35]
\clip (-8, -6) rectangle (8, 6);
\draw (0,0) circle (5);
\filldraw[fill=black] (0,0) circle (.05);
\draw (0.4,-0.08) node {${}^0$}; 
\draw (-2.5,2.5) node {$B^\lambda$}; 
\draw (-1,0.3) node {\small $\lambda$};
\draw (-4.9,0.85) -- (4.9,0.85); 
\draw[<->] (0,0) -- (0,0.85); 
% \draw[->] (0,0) -- (0,0.85);
\path[loosely dotted, fill=white!60!black,opacity=0.5, very thin] 
(4.9,0.85) arc[start angle = 10, end angle =170, radius = 5] -- cycle;
\end{scope}
\end{tikzpicture}
\end{minipage}

\caption{%
An example of partition with the half-line property. The grey regions represent\,\, $A_i \cap \{\xi\in B_1:\,\xi\cdot\nu_i > \lambda\}$
}
\label{fig:due-cerchi-allineati}
\end{figure}

The proof of the above inequality is a difficult geometric problem that we solve completely in Section~\ref{sec:maingeo}, see Theorem~\ref{thm:maingeo}. Let us just mention here that the convexity of $\Co$ is used only to prove that the subdifferentials satisfy the following {\em half-line property}: if $\xi\in A_i$ then $\xi+t\nu_i\in A_i$ for all $t>0$. The arguments of Section~\ref{sec:maingeo} show that in fact~\eqref{introcr1} hold for any finite partition of $\R^N$ into essentially disjoint sets satisfying the above half-line property.

\section{Preliminary results}

In this section we set up the tools and show some preparatory results that we will need later for the ABP argument. To this aim, in Section~\ref{sec:neu} we establish the crucial viscosity supersolution property for the variational solutions of the Neumann problem~\eqref{eq:neumann-0}. In Section~\ref{sec:subdif} we introduce a useful notion of restricted subdifferential which enables us to reduces the inequality~\eqref{eq:crucial-est} to a property of the convex set $\Co$. 
%As we already mentioned, we call it $\lambda$-ABP property and give its definition in Section~\ref{sec:lambdaABP} (Definition~\ref{def:ABP}). At the end of the section we establish the relative isoperimetric inequality for the $\lambda$-capillary functional outside the convex sets that satisfy such a property.

As we mentioned in the introduction, we will first prove the relative isoperimetric inequality for regular sets. In order to emphasize this we always denote $E = \Omega$ when the set is assumed to be regular, i.e., it satisfies 
\beq\label{Coreg}
\text{$\Co\subset\R^N$ is a closed convex set of class $C^{2}$} 
\eeq
and \beq\label{Omreg} 
\begin{split}
&\text{$\Om\subset\R^N\setminus\Co$ is a bounded Lipschitz open set such that $\Sigma:=\pa\Om\setminus\Co$ } \\
&\qquad\qquad\qquad \text {is a $(N-1)$-manifold with boundary of class $C^{2}$}.
\end{split}
\eeq 
We call the boundary $\Sigma$ the \emph{free interface} and denote $\Gamma:=\pa\Om\cap \Co$, which we call the \emph{wetted region} which is also an embedded $C^2$-regular $(N-1)$-manifold with boundary. Note that $\Gamma$ and $\Sigma$ share the same boundary, which we denote by $\gamma:=\overline \Sigma\cap \Co$ and which by the assumption is a $(N-2)$-manifold of class $C^{2}$. We call $\gamma$ the \emph{contact set} of $\Sigma$ with $\Co$. Moreover, we will throughout the section assume $|\Omega| = |B^\lambda|$ if not otherwise mentioned. 

We will denote by $\nu_\Om$ and $\nu_\Co$ the outer unit normal to $\pa \Om$ and to $\pa \Co$, respectively. We also denote by $\nu_\Sigma=\nu_\Om$ the outer unit normal field on $\Sigma$, which by our assumptions admits a continuous extension at $\gamma$, still denoted by $\nu_\Sigma$. We also set $\nu_\Gamma=-\nu_\Co$ on $\Gamma$. Note that the Lipschitz regularity of $\Om$ yields 
\[
\nu_\Gamma \cdot \nu_\Sigma> -1\qquad\text{on }\gamma\,.
\]
Finally, we define the $\eps$-neighborhood of a generic set $F \subset \R^N$ as $(F)_\eps:=B_\eps + F = \{ x \in \R^N : \text{dist}(x,F)< \eps\}$. 

\subsection{The Neumann problem}\label{sec:neu}

In this section we consider the Neumann problem~\eqref{eq:neumann-0} under the assumptions ~\eqref{Coreg} and~\eqref{Omreg}. In addition, we assume that $\Omega$ is connected. Note that even in this case $\Omega$ is still merely a Lipschitz domain and therefore the high regularity of $u$ up to the boundary is not granted. However, it turns out that we only need the solution to attain the boundary values in the viscosity sense, for which H\"older continuity up to the boundary is enough. To this aim we consider the variational solution of the problem~\eqref{eq:neumann-0} which by definition is a function $u\in H^1(\Om)$ such that it holds 
\[
\int_{\Omega} \nabla u \cdot \nabla \varphi\,dx= -\int_\Om c \varphi\,dx+\int_{\pa\Om}g\varphi\,d\H^{N-1} 
\]
for all $\varphi\in H^1(\Omega)$, where $c$ is given by~\eqref{eq:compatibility} and
\beq\label{g}
g \equiv 1\text{ on } \Sigma \quad\text{and}\quad g \equiv -\lambda \quad\text{on } \Gamma \setminus \gamma\, .
\eeq
Since $\Omega$ is bounded and Lipschitz regular, $g$ is bounded and we have the compatibility condition~\eqref{eq:compatibility} which can be rewritten as
\[
c |\Omega | =\int_{\pa\Omega}g\,d\H^{N-1}\, ,
\] 
there exists a unique (up to an additive constant) variational solution of~\eqref{eq:neumann-0}. Moreover, by standard elliptic regularity theory the variational solution is H\"older continuous up to the boundary, see for instance~\cite{Nittka}.
%In fact, the H\"older continuity holds even in more general setting. Given $f\in L^\infty(\Omega)$, $g\in L^\infty(\pa\Omega)$ we consider the following Neumann problem
%\begin{equation} \label{eq:neumanngeneral}
%\begin{cases}
%&-\div(A(x)\nabla u) = f\quad \text{in } \, \Omega \\
%& A(x)\nabla u\cdot\nu = g \quad \text{on } \, \pa\Omega,
%\end{cases}
%\end{equation}
%where $A(x)$ is a uniformly elliptic matrix in $\Omega$ with measurable coefficients, that is there exist $0<a<b$ such that for a.e. $x\in\Omega$
%\beq\label{unifEll}
%a|\xi|^2\leq A(x)\xi\cdot \xi\leq b|\xi|^2 \qquad \text{ for all } \,\xi\in\R^N.
%\eeq
%We have the following result.
%\begin{proposition}\label{th:reg}
%Let $\Om\subset\R^N$ be a bounded open set with Lipschitz boundary. There exist $\alpha \in(0,1)$ and $C>0$ such that, if $u$ a solution of~\eqref{eq:neumanngeneral} with $\medintinrigo_{\Omega}u=0$, then
%\beq\label{reg0}
%\|u\|_{C^{\alpha}(\overline\Omega)}\leq C.
%\eeq
%\end{proposition}
%This theorem can be proved by showing that $u$ belongs to a suitable De Giorgi class up to the boundary $\pa\Omega$, see the Appendix for an elementary proof. Note that $\alpha$ and the constant $C$ in~\eqref{reg0} depend only on $\Omega$, on the constants in the trace Theorem, in the Sobolev embedding Theorem and in the Poincar\'e inequality for functions with zero average in $\Om$. For more refined estimates

Let us proceed to the notion of viscosity solution. In fact, since we need only the concept of viscosity supersolution for the ABP-argument, we reduce to that. Here is the definition we need. 
\begin{definition}
\label{def:visco}
A lower semicontinuous function $u : \overline \Omega \to \R$ is a \emph{viscosity supersolution} of~\eqref{eq:neumann-0} if whenever $u-\varphi$ has a local minimum at $x_0\in \overline\Om$ for $\varphi \in C^2(\R^N)$, then
\[
\begin{cases}
-\Delta \varphi(x_0) \geq -c & \text{if } \, x_0 \in \Omega\,,\\
\partial_{\nu_\Sigma}\varphi(x_0) -1 \geq 0 & \text{if } \, x_0 \in \Sigma \setminus \gamma\,,\\
\partial_{\nu_\Gamma} \varphi(x_0)+\lambda \geq 0 & \text{if } \, x_0 \in \Gamma\setminus \gamma\,,\\
\max\{\partial_{\nu_\Sigma}\varphi(x_0) -1, \partial_{\nu_\Gamma}\varphi(x_0)+\lambda \}\geq 0 & \text{if }x_0\in \gamma\,,
\end{cases}
\]
where $\partial_{\nu_\Sigma}\varphi(x_0) = \nabla \varphi(x_0) \cdot \nu_{\Sigma}(x_0)$ and $\partial_{\nu_\Gamma}\varphi(x_0) = \nabla \varphi(x_0) \cdot \nu_{\Gamma}(x_0)$. 
\end{definition}

We can now prove the main result of the section.
\begin{proposition}\label{th:corner}
Let $\Om$, $\Co$ be as in~\eqref{Coreg} and~\eqref{Omreg}. Then the variational solution of~\eqref{eq:neumann-0} is a viscosity supersolution in the sense of Definition~\ref{def:visco}.
\end{proposition}
\begin{proof} 

Since the equation is satisfied classically in $\Om$ and the Neumann boundary conditions are achieved in a classical sense at $\pa\Om\setminus \gamma$, it will be enough to check the property on $\gamma$. To this aim, assume that $\varphi\in C^2(\R^N)$ and $x_0\in \gamma$ are such that $(u-\varphi)(x)\geq 0$, with equality achieved only at $x_0$. We start by showing that 
\beq\label{cond1}
\max\{-\Delta \varphi(x_0)+c, \pa_{\nu_\Sigma}\varphi(x_0)-1,\pa_{\nu_\Gamma}\varphi(x_0)+\lambda\}\geq 0\,.
\eeq
We argue by contradiction, assuming that 
\[
\max\{-\Delta \varphi(x_0)+c, \pa_{\nu_\Sigma}\varphi(x_0)-1,\pa_{\nu_\Gamma}\varphi(x_0)+\lambda\}< 0\,.
\]
By continuity we may find a small ball $B_r(x_0)$ such that 
\[
-\Delta \varphi+c<0 \quad \text{in }B_r(x_0)\qquad\text{and}\qquad \pa_{\nu}\varphi-g < 0\quad\text{ on } (\pa\Omega\cap B_r(x_0))\setminus \gamma\,,
\]
with $g$ defined in~\eqref{g}.
% Clearly, by a similar (in fact simpler) argument we can also show that $-\Delta \varphi(x_0)+c\geq 0$ if $x_0\in \Om$ and 
%\beq\label{cond2}
% \begin{split}
% \max\{-\Delta \varphi(x_0)+c, \pa_{\nu_\Sigma}\varphi(x_0)-1\}\geq 0 & \qquad \text{if }x_0\in \Sigma,\\
% \max\{-\Delta \varphi(x_0)+c, \pa_{\nu_\Gamma}\varphi(x_0)+\lambda\}\geq 0 & \qquad \text{if }x_0\in \Gamma\setminus\gamma.
% \end{split}
%\eeq
Then, setting $w:=u-\varphi$, $h:=-c+\Delta \varphi$, $f=\pa_{\nu}(u-\varphi)=g-\pa_{\nu}\varphi$, we have that $w$ is a variational solution of
\[
\begin{cases}
-\Delta w= h & \text{in }\Om \cap B_r(x_0)\,,\\
\pa_\nu w=f & \text{ in } \pa\Omega \cap B_r(x_0)\,,
\end{cases}
\] 
that is, 
\beq\label{debole}
\int_{\Omega}\nabla w\cdot \nabla\psi\, dx=\int_{\Om}h\psi\, dx+\int_{\pa \Om} f\psi\, d\H^{N-1}
\eeq
for all $\psi\in H^1(\Om)$ with $\psi =0$ in $\Om\setminus B_{r}(x_0)$. Let us now choose $\psi:=\min\{w-\eps, 0\}$ and note that for $\eps>0$ small enough 
\[
\psi=\min\{w-\eps, 0\}=0\quad\text{in }\Om\setminus B_{r}(x_0)\,.
\]
Then, \eqref{debole} combined with the fact that $h>0$ in $\Om\cap B_r(x_0)$ and $f>0$ in $\pa \Om\cap B_r(x_0)$, yield
\[
\int_{\Omega}\big|\nabla \big(\min\{w-\eps, 0\}\big)\big|^2 dx= \int_{\Om\cap B_r(x_0)}h\psi\, dx+\int_{\pa \Om\cap B_r(x_0)} f\psi\, d\H^{N-1}\leq 0\,
\]
and in turn $\min\{w-\eps, 0\}=0$ in $\Om$. This is impossible since $w-\eps<0$ in a neighborhood of $x_0$. Thus~\eqref{cond1} is established.

The inequality~\eqref{cond1} is not good enough, since we only want information on the boundary. We thus claim that in fact
\beq\label{cond3}
% \begin{split}
\max\{ \pa_{\nu_\Sigma}\varphi(x_0)-1,\pa_{\nu_\Gamma}\varphi(x_0)+\lambda\}\geq 0\,.
% &\qquad \text{if $x_0\in\gamma$}, \\
% \pa_{\nu_\Sigma}\varphi(x_0)-1\geq 0 & \qquad \text{if }x_0\in \Sigma,\\
% \pa_{\nu_\Gamma}\varphi(x_0)+\lambda\geq 0 & \qquad \text{if }x_0\in \Gamma\setminus\gamma.
% \end{split}
\eeq
To this aim we observe that for any $x_0\in\gamma$ there exists a ball $B_r(\bar x)\subset\R^N\setminus\overline\Om$ with $x_0\in\pa B_r(\bar x)$ (it is enough to take a ball contained in $\Co$ and tangent to $x_0$, which is possible by the $C^{2}$ assumption on $\Co$).
By translating and dilating we may assume for simplicity that $\bar x=0$ and that $r=1$. We perturb the test function $\varphi$ by a functions $\psi_q\in C^2(\R^n \setminus \{ 0 \} )$ that we define as 
\[
\psi_q(x) = \frac{1}{q^{3/2}}(|x|^{-q} - 1)
\]
where $q>0$ is a large number to be chosen. Then by the exterior ball condition we have that $\psi_q(x) \leq 0$ for $x \in \Omega$ while since $x_0 \in \pa B_1$ it holds $\psi_q(x_0)= 0$. Moreover by a direct computation we see that 
\[
\Delta \psi_q(x_0) \geq \frac{\sqrt{q}}{2} \quad \text{and} \quad |\nabla \psi_q(x_0)|= \frac{1}{\sqrt{q}},
\]
for $q$ sufficiently large.
We define a new test function 
\[
\varphi_q(x) = \varphi(x) + \psi_q(x)\,.
\]
By construction it holds $\varphi_q \leq\varphi$ in $\Omega $ and $\varphi_q(x_0) = \varphi(x_0)$, hence $x_0$ is still a local minimum for $u-\varphi_q$. Thus
\[
-\Delta \varphi_q(x_0)+c\leq -\frac{\sqrt{q}}{2}-\Delta \varphi(x_0)+c<0
\]
when $q$ is large. Therefore, for $q$ large, from~\eqref{cond1} we obtain~\eqref{cond3} with $\varphi$ replaced by $\varphi_q$. Finally, letting $q\to\infty$, since $\nabla\varphi_q(x_0)\to\nabla\varphi(x_0)$ we obtain~\eqref{cond3}. This concludes the proof.
\end{proof}

\subsection{Subdifferentials and restricted subdifferentials}\label{sec:subdif}
We need some notation in order to proceed. Given $X\subset\R^N$, a function $u : X \to \R$, a subset $Y\subset X$ and a point $x\in Y$ we define the subdifferential $J_Yu(x)$ as
\[
J_Yu(x):=\big\{\xi\in\R^N:\,u(y) - u(x) \geq \xi \cdot (y-x) \,\,\, \text{for all}\, \, y \in Y\big\}\,.
\]
We note that $Y$ may even be a discrete set. 

\begin{remark}\label{rm:ovv}
Note that if $Y$ is compact and $u$ is continuous in $Y$, then
\[
\bigcup_{x\in Y}J_Y u(x)=\R^N\,.
\]
Indeed, for any $\xi\in \R^N$, we may find $c>0$ so large that $\displaystyle\sup_{x\in Y}\big(-c+\xi\cdot x-u(x))<0$. Setting 
\[
t_0:=\sup\{ t>0:\, -c+\xi\cdot x+t<u(x)\text{ for all $x\in Y$}\}\,,
\]
we clearly have $-c+\xi\cdot \bar x+t_0=u(\bar x)$ for some $\bar x\in Y$ and $-c+\xi\cdot y+t_0\leq u(y)$ for all $y\in Y$; that is, $\xi\in J_Y u(\bar x)$.

Moreover, note that for any $x,x'\in Y$ the intersection $J_Y(x)\cap J_Y(x')$ is contained in a hyperplane orthogonal to $x'-x$.
\end{remark}
%If $u$ is the solution of~\eqref{eq:neumann} then all vectors $\xi \in \R^{N}$ can be divided into three classes
%\[
%\begin{split}
% \begin{cases}
%&\xi \in Ju(x) \, \quad \text{for } \, x \in \Omega \\
%&\xi \in Ju(x) \, \quad \text{for } \, x \in \Sigma\setminus \Gamma \\ 
%&\xi \in Ju(x) \, \quad \text{for } \, x \in \Gamma. 
%\end{cases}
%\end{split}
%\]
%In the following I will assume that $\Co $ is smooth and uniformly convex and $u$ is $C^2$. 
Let $\Om$, $\Gamma$ and $\Sigma$ be as in~\eqref{Coreg} and~\eqref{Omreg}.
% Recall that the crucial inequality~\eqref{eq:crucial-est} involves the set $\mathcal A_u $ which we define for any given function $u:\overline\Om\to\R$ as
%\beq\label{alambdau}
%\mathcal A_u := \bigcup_{x\in\Omega} \{\xi \in \R^{N} : \xi \in J_{\overline\Om}\,u(x) \}.
%\eeq
As we will see, in order to prove the crucial inequality~\eqref{eq:crucial-est}, it turns out that all relevant information of $u$ (the solution of~\eqref{eq:neumann-0}) is contained in its restriction to $\Gamma$. In fact, it has to do with the boundary conditions rather than with the PDE. To this aim, we define the following union of \emph{restricted subdifferentials} for functions defined on compact subsets of the boundary of $\Co$, $v : K \subset \partial \Co \to \R$ 
\beq\label{blambdau}
\mathcal B_{v}^\lambda :=\bigcup_{x\in K} \{ \xi \in J_K v(x) : \,|\xi|<1\text{ and }\xi\cdot\nu_\Co(x)> \lambda \}\,,
\eeq
where $\lambda \in (-1,1)$. 
%If $\Om$ and $u:\overline\Om\to\R$ are as above then we denote the restriction of $u$ to $\Gamma$ by $u_\Gamma$ and by the previous notation we have 
%\beq\label{blambdau2}
%\mathcal B_{u_{\Gamma}}^\lambda =\bigcup_{x\in \Gamma} \{ \xi \in J_{\Gamma} u(x) : \,\xi\cdot\nu_\Co(x)> \lambda \}.
%\eeq
The following crucial lemma allows us to rewrite the information on the Neumann boundary conditions in ~\eqref{eq:neumann-0} in terms of a condition on the restricted subdifferentials just introduced. 
\begin{lemma}
\label{lem:u-gamma}
Let $\Om$, $\Sigma$ and $\Co$ be as in~\eqref{Coreg} and~\eqref{Omreg}. Let $u$ be the variational solution of~\eqref{eq:neumann-0}. Then, denoting by $u_\Gamma$ the restriction of $u$ to $\Gamma$, it holds 
\[
\mathcal B_{u_{\Gamma}}^\lambda \subset \nabla u(\widehat \Om)
\]
where $\mathcal B_{u_{\Gamma}}^\lambda$ is as in ~\eqref{blambdau}, with $v=u_\Gamma$, and $\widehat \Om$ is defined in~\eqref{omegahat}.
\end{lemma} 

\begin{proof}
Recall that $u$ is continuous up to the boundary and by Proposition~\ref{th:corner} it is a viscosity supersolution of~\eqref{eq:neumann-0} in the sense of 
Definition~\ref{def:visco}.
Fix $\xi \in \mathcal B^\lambda_{u_\Gamma}$. Then $\xi\in J_\Gamma u(x_0)$, for some $x_0\in\Gamma$, $|\xi|<1$ and $\xi\cdot\nu_\Co(x_0)>\lambda$. We need to show that $\xi \in \nabla u(\widehat \Om)$, which by~\eqref{omegahat} means that $\xi \in J_{\Om}u(\bar x)$ for some $\bar x \in \Om$.

First we claim that $\xi\not\in J_{\overline\Om}u(x_0)$. Indeed, assume the opposite, that is
\[
u(y)\geq u(x_0) + \xi \cdot (y- x_0)=: \varphi(y) \qquad \text{for all } \, y \in \overline \Omega\,.
\]
In particular, $x_0$ is the minimum point of $u - \varphi$ and since $u$ is a viscosity supersolution and $\nabla \varphi(x_0) = \xi$, see Definition~\ref{def:visco}, we have
\[
\begin{cases}
\xi \cdot \nu_\Gamma( x_0)+\lambda \geq 0 & \text{if } \, x_0\in \Gamma\setminus \gamma\,,\\
\max\{ \xi \cdot \nu_\Sigma( x_0) -1, \xi \cdot \nu_\Gamma( x_0)+\lambda \}\geq 0 & \text{if } x_0\in \gamma\,.
\end{cases}
\]
From the above condition, since $|\xi|<1$ we have that $ \xi \cdot \nu_\Gamma( x_0)+\lambda\geq0$, which is impossible since $\xi \in \mathcal B^\lambda_{u_\Gamma}$ and so 
$\xi\cdot\nu_\Gamma( x_0)=-\xi\cdot\nu_\Co( x_0)<-\lambda$. Therefore $\xi\not\in J_{\overline\Om}u(x_0)$. 

By the above it holds that the inequality $u(y)\geq \varphi(y)$ is not true for all $y\in\overline\Omega$. This means that $\bar c=\min_{y\in\overline\Omega}\,(u(y)-\varphi(y))<0$. In turn we have that the graph of $\varphi+\bar c$ touches from below the graph of $u$ at some point $\bar x\in\overline\Om$. Clearly $\bar x\not\in\Gamma$, since $\xi\in J_\Gamma u(x_0)$ and $\bar c<0$ imply $\varphi(y)+\bar c<u(y)$ for all $y\in\Gamma$. On the other hand, if $\bar x\in \Sigma$, again by the supersolution property, we would have that $\xi\cdot\nu_\Sigma(\bar x)\geq1$, which is impossible as $|\xi|<1$. Therefore $\bar x\in\Omega$, which implies $\xi=\nabla u(\bar x)$, as desired. 
\end{proof}

\section{The main geometric estimate}\label{sec:maingeo}

By Lemma~\ref{lem:u-gamma} we know that if 
$
|\mathcal B_{u_{\Gamma}}^\lambda |\geq |B^\lambda|,
$
where $u_\Gamma$ is the restriction of $u $ on $\Gamma$ and $\mathcal B_{u_{\Gamma}}^\lambda$ is as in ~\eqref{blambdau}, with $v=u_\Gamma$, then we have inequality~\eqref{eq:crucial-est}. 
As we will show in this section, the above estimate on restricted subdifferentials holds in fact for generic continuous functions $v: K \to \R$ , where $K$ is any compact subset of $\Co$ and not just for $u_\Gamma$. This in turn means that such a property pertains solely to the convexity of the set $\Co$. Precisely, we will show the following. 
\begin{theorem}\label{thm:maingeo}
Let $\Co\subset\R^N$ be any closed convex set with nonempty interior and of class $C^1$ and let $\lambda\in (-1, 1)$. If $K\subset\pa\Co$ is compact and $v:K\to\R$ is any continuous function, then 
\beq\label{eq:maingeo2}
|\mathcal B^\lambda_{v} |\geq|B^\lambda|\,,
\eeq
where $B^\lambda=\{x\in B_1:\,x\cdot e_N>\lambda\}$ and $\mathcal B^\lambda_{v}$ is defined in~\eqref{blambdau}.
\end{theorem}
It turns out that the crucial property in establishing the above estimate is the following \emph{half-line property} satisfied by the subdifferentials of functions defined on subsets of $\pa \Co$, which is a consequence of the mere convexity of $\Co$. 
\begin{lemma}
\label{lem:v-gamma}
Let $\Co\subset\R^N$ be a closed convex set with nonempty interior and of class $C^1$, let $K \subset \partial \Co$ and let $v : K \to \R$. 
If $\xi\in J_K v(x)$ for some $x\in K$, then 
\[
\xi+t\nu_\Co(x)\in J_K v(x) \qquad \text{ for all } \, t>0\,.
\]
\end{lemma} 
\begin{proof}
If $\xi\in J_K v(x)$, then
\[
v(y) - v(x) \geq \xi \cdot (y-x) \qquad \text{for all } \, y \in K\,.
\]
By convexity it holds $\nu_\Co(x) \cdot (y - x) \leq 0$ for all $y \in K$. Therefore for any $t>0 $ it holds 
\[
v(y) - v(x) \geq (\xi + t \nu_\Co(x) ) \cdot (y-x) \qquad \text{for all } \, y \in K\,,
\]
that is $\xi + t \nu_\Co(x) \in J_K v(x)$. 
\end{proof}
The next useful observation is that by discretization we can reduce the proof of~\eqref{eq:maingeo2} to the case of functions~$v$ defined on finite subsets of $\pa \Co$, as shown in the following lemma. 
To this aim we recall that a sequence $\{C_n\}$ of closed sets of $\R^N$ converges in the {\em Kuratoswki sense} to a closed set $ C$ if the following conditions are satisfied:
\begin{itemize}
\item[(i)] if $x_n\in C_n$ for every $n$, then any limit point of $\{x_n\}$ belongs to $ C$;
\item[(ii)] any $x\in C$ is the limit of a sequence $\{x_n\}$ with $x_n\in C_n$.
\end{itemize}
One can easily see that $C_n\to C$ in the sense of Kuratowski if and only if dist$(\cdot, C_n)\to$ dist$(\cdot, C)$ locally uniformly in $\R^N$. 

\begin{lemma}\label{lem:ABP}
Let $\Co\subset\R^N$ be a closed convex set with nonempty interior and with $C^1$ boundary, and fix $\lambda\in (-1, 1)$. Assume that~\eqref{eq:maingeo2} holds for all $v:K\to \R$, whenever $K\subset\pa\Co$ is finite. Then, \eqref{eq:maingeo2} holds for all continuous $v:K\to \R$, whenever $K\subset\pa\Co$ is compact.
\end{lemma}
\begin{proof}
Fix a compact subset $K$ of $\pa\Co$ and a continuous function $v : K \to \R$. Let $K_n\subset K$ be a sequence of discrete sets such that $K_n\to K$ in the Kuratowski sense. We define the function $v_n : K_n \to \R$ as the restriction of $v$ on $K_n$, i.e., $v_n(x)= v(x)$ for $x\in K_n$. 

We claim that 
\beq\label{eq:bella1.5}
\Chi{\mathcal B^{\lambda}_{v}}\geq \limsup_n \Chi{\mathcal B^{\lambda'}_{v_n}}\quad \text{pointwise in $B_1$ }\qquad \forall \lambda'\in (\lambda, 1)\,.
\eeq
We argue by contradiction, by assuming that for some $\lambda'\in ( \lambda, 1)$ and for some $\xi\in B_1\setminus \mathcal B^{\lambda}_{v}$, along a (non relabelled) subsequence $\xi\in \mathcal B^{\lambda'}_{v_n}$; i.e., there exist $x_n\in K_n$ such that $\xi\in J_{K_n}v_n(x_n)$ and $\xi\cdot\nu_{\Co}(x_n)>\lambda'$. Up to extracting a further (non-relabelled) subsequence we may assume that $x_n\to x$ for some $x\in K$. Clearly, we also have $\xi\cdot\nu_{\Co}(x_n)\to \xi\cdot\nu_{\Co}(x)$, so that 
\beq\label{eq:bella3}
\xi\cdot\nu_{\Co}(x)\geq\lambda'\,.
\eeq
Fix now $y\in K$. Then, by Kuratowski convergence there exists a sequence $y_n\to y$ such that $y_n\in K_n$ for all $n$. Since $\xi\in J_{K_n}v_n(x_n)$, we have 
\[
v_n(y_n)\geq v_n(x_n)+\xi\cdot (y_n-x_n)\,. 
\]
Passing to the limit and noticing that $v_n(x_n)\to v(x)$ and $v_n(y_n)\to v(y)$, we infer $\xi\in J_K v(x)$. Recalling~\eqref{eq:bella3}, we get in particular $\xi\in \mathcal B^{\lambda}_{v}$, which is a contradiction. This establishes~\eqref{eq:bella1.5}, which in turn implies 
\beq\label{eq:bella2}
|\mathcal B^{\lambda}_{v}|\geq \limsup_n |\mathcal B^{\lambda'}_{v_n}| \qquad \forall \lambda'\in (\lambda, 1)\,,
\eeq
by Fatou's lemma. 

By assumption for any $n$ we have the estimate 
\beq\label{eq:bella}
|\mathcal B^\lambda_{v_n} |\geq|B^{\lambda}|. 
\eeq 
Let us then prove that there exists a constant $C$ independent of $n$ such that 
\beq\label{eq:brutta}
|\mathcal B^{\lambda'}_{v_n} | \geq |\mathcal B^\lambda_{v_n} | - C(\lambda' - \lambda) \qquad\text{for all $\lambda'\in (\lambda, \lambda+\delta)\,$,}
\eeq
where $\delta=(1-\lambda)/4$. For any $n$ and $t \in (0, 1-\lambda)$ we define $A_t$ (dropping the dependence on $n$) as
\[
A_t :=\bigcup_{x\in K_n} \{ \xi \in J_{K_n} v_n(x) : \,\xi\cdot\nu_\Co(x)= \lambda+t \}
\]
and we set for $r>0$
\[
\mathcal B_{v_n, r}^{t} :=\bigcup_{x\in K_n} \mathcal B_{v_n, r}^{t}(x)=\bigcup_{x\in K_n}\{ \xi \in J_{K_n} v_n(x) : \,|\xi|<r\text{ and }\xi\cdot\nu_\Co(x)> t\}\,.
\]
For any $x\in K_n$ the set $ \mathcal B_{v_n, r}^{t}(x)$ is the intersection of a fixed bounded set with the half space $\{\xi\cdot\nu_\Co(x)> t\}$. As a consequence the function $t \mapsto |\mathcal B^{\lambda+t}_{v_n, r} |$ is decreasing and Lipschitz continuous (with constant possibly depending on the cardinality of $K_n$). Moreover, at points of differentiability it holds 
\beq\label{eq:brutta2}
\frac{d}{d t} |\mathcal B^{\lambda+t}_{v_n, r} | = - \H^{N-1}(A_t \cap B_r) 
\eeq
for all $r>0$. 

The idea is to show that in fact the Lipschitz constant is independent of $n$. 
To this aim, we fix $x \in K_n$ and study the set 
\[
A_{t}(x) = \{ \xi \in J_{K_n} v_n(x) : \,\xi\cdot\nu_\Co(x)= \lambda+t \}\,. 
\]
Assume that $\xi \in A_{t}(x)$, so in particular $ \xi \in J_{K_n} v_n(x) $. By Lemma~\ref{lem:v-gamma} for every $s>0$ it holds 
$\xi + s \nu_\Co(x) \in J_{K_n} v_n(x)$, so $\xi + s \nu_\Co(x)\in A_{t+s}(x)$. Therefore for every $s>0$ it holds 
\[
A_{t}(x) + s \nu_\Co(x) \subset A_{t+s}(x) \, .
\]
Then, it holds $\H^{N-1}( A_{t}(x) \cap B_{1+t}) \leq \H^{N-1}( A_{t+s}(x) \cap B_{1+t+s})$. Therefore, the function 
\[
t \mapsto \H^{N-1}(A_t \cap B_{1+t}) = \sum_{x \in K_n} \H^{N-1}( A_{t}(x) \cap B_{1+t})
\]
is non-decreasing. Notice that the equality holds true since for any $x,x'\in K_n$ one has $\H^{N-1}(A_t(x)\cap A_t(x'))=0$. Indeed by Remark~\ref{rm:ovv} $A_t(x)\cap A_t(x')\subset J_{K_n}v_n(x)\cap J_{K_n}v_n(x')$ is contained in a hyperplane orthogonal to $x'-x$ and neither $\nu_{\Co}(x)$ nor $\nu_{\Co}(x')$ can be parallel to $x'-x$. Integrating~\eqref{eq:brutta2} we have 
\[
\int_{\delta}^{2\delta} \H^{N-1}(A_t \cap B_{1+t})\, dt \leq \int_{\delta}^{2\delta} \H^{N-1}(A_t \cap B_{1+2\delta})\,dt =- \int_{\delta}^{2\delta} \frac{d}{dt} |\mathcal B^{\lambda+t}_{v_n, 1+2\delta}| \, dt \leq |B_2|\,.
\]
By the mean value theorem there is $\hat t \in [\delta,2\delta]$ such that $\H^{N-1}(A_{\hat{t}} \cap B_{1+\hat{t}}) \leq C=2^{N+2}\omega_N/(1-\lambda)$.
In turn, using the monotonicity obtained above, we have 
\[
\H^{N-1}(A_t \cap B_{1}) \leq C \qquad \text{for all }\, t \in [0,\delta]\,.
\]
The claim~\eqref{eq:brutta} then follows by integrating~\eqref{eq:brutta2} from $t=0$ to $t = \lambda'-\lambda$. 
Finally the statement of the lemma follows from~\eqref{eq:bella2}, \eqref{eq:bella} and~\eqref{eq:brutta} and letting $\lambda' \to \lambda$. \end{proof}

Thanks to the previous result we are reduced to consider discrete functions defined on finite subsets $K\subset\pa\Co$. In this situation, there are only finitely many subdifferentials $J_Kv(x)$ and it is also immediate that they are essentially disjoint, i.e., they have disjoint interiors. Hence, by Remark ~\ref{rm:ovv} and Lemma~\ref{lem:v-gamma} the subdifferentials $J_Kv(x)$ for $x \in K$ form a finite partition of the space $\R^N$ made of sets having the half-line property. As we will see, Theorem~\ref{thm:maingeo} follows from these properties only, but the argument is rather involved in the general case. However, the situation becomes particularly simple in the Choe-Ghomi-Ritor\'e case $\lambda=0$.

\begin{proof}[\bf Proof of Theorem~\ref{thm:maingeo} for $\lambda=0$]
By Lemma~\ref{lem:ABP} it is enough to consider functions defined on finite sets. Let $K=\{x_1, \dots, x_n\}$ be any finite subset of $\pa\Co$ and let $v:K\to \R$. Recall that $\displaystyle\bigcup_{i=1}^nJ_K v(x_i)=\R^N$ and that the sets $J_K v(x_i)$ have disjoint interiors. Moreover Lemma~\ref{lem:v-gamma} implies that they have the half-line property
\beq\label{recall}
\xi\in J_K v(x_i)\Longrightarrow \xi+t\nu_{\Co}(x_i)\in J_K v(x_i) \text{ for all $t>0\,$.}
\eeq

% Moreover, it is easy to check that the sets $J_\Gamma u(x_i)$ have disjoint interiors. Indeed, $\xi\in \mathrm{Int\,}J_\Gamma u(x_i)$ if and only if the graph of the function $x\mapsto u(x_i)+\xi\cdot(x-x_i)$ touches the graph of $u$ from below only at $x_i$. 

Now, up to a set of Lebesgue measure zero, we may split $J_K v(x_i)=J_K v(x_i)^+\cup J_K v(x_i)^-$, where
\[
J_K v(x_i)^\pm:=\{\xi\in J_K v(x_i):\, \pm\, \xi\cdot\nu_{\Co}(x_i)>0 \}\,.
\]
Let us then fix $\xi\in J_K v(x_i)^-$ and denote $\tau = -\xi \cdot \nu_{\Co}(x_i)>0$. Then by~\eqref{recall} the symmetric point $\hat\xi = \xi + 2 \tau \nu_{\Co}(x_i)$ belongs to $J_K v(x_i)^+$ and has the same norm as $\xi$. Hence $|J_K v(x_i)^+\cap B_1|\geq \frac12 |J_K v(x_i)\cap B_1|$. In turn,
\beq\label{0ABPcrucial}
|\mathcal B^0_{v}|=\sum_{i=1}^n|J_K v(x_i)^+\cap B_1|\geq \frac12\sum_{i=1}^n |J_K v(x_i)\cap B_1|=\frac12 |B_1|\,,
\eeq
which is the desired estimate.
\end{proof}

\begin{remark}\label{rm:choeghomiritore}
Combining the previous proof with the ABP argument skeched in the introduction and rigorously developed in the next section, we recover an ABP-proof of the relative isoperimetric inequality outside convex sets obtained by Choe, Ghomi and Ritor\'e in~\cite{ChGhoRi07}. Note that an ABP-argument for the same inequality has been already provided in~\cite{LiuWangWeng}. However, in their argument our crucial estimate~\eqref{0ABPcrucial} is replaced by a geometric estimate based on normal cones and inspired by the techniques in~\cite{ChGhoRi06}, see~\cite[Proposition~2.4]{LiuWangWeng}.
\end{remark}

We turn to the general case. To this aim let $K=\{x_1, \dots, x_n\}$ be any finite subset of $\Co$ and let $v:K\to \R$ be any function. It is obvious that Theorem~\ref{thm:maingeo} follows if we show that 
\[
\H^{N-1}\Big(\bigcup_{x\in K} \{\xi\in\pa B_\varrho\cap J_Kv(x) : \, \xi\cdot\nu_\Co(x)> \lambda \}\Big)
\geq \H^{N-1}\big(\{\xi\in \pa B_\varrho: \xi\cdot e_N>\lambda\}\big)\,, 
\]
for all $\varrho \in(0,1)$. By rescaling, this inequality is equivalent to
\beq\label{prop:lambdaABP2.1}
\H^{N-1}\Big(\bigcup_{x\in K} \{\xi\in\pa B_1\cap J_Kv_\varrho(x) : \, \xi\cdot\nu_\Co(x)> \lambda' \}\Big)
\geq \H^{N-1}\big(\{\xi\in \pa B_1: \xi\cdot e_N>\lambda'\}\big)\,,
\eeq
where we have set $\lambda'=\frac{\lambda}{\varrho}\in\R$, $v_\varrho=\frac{v}{\varrho}$. Note that ~\eqref{prop:lambdaABP2.1} is trivially satisfied if $\lambda'\geq1$ or $\lambda'\leq-1$. 

Therefore, we are ultimately bound to show that for any function $v:K\to\R$ and any finite set $K\subset\pa\Co$ we have
\beq\label{tesi}
\H^{N-1}\Big(\bigcup_{x\in K} \{\xi\in\pa B_1\cap J_Kv(x) : \, \xi\cdot\nu_\Co(x)> \lambda\}\Big)
\geq \H^{N-1}\big(\{\xi\in \pa B_1: \xi\cdot e_N>\lambda\}\big)\,.
\eeq
In the case $\lambda<0$, it will be in fact be convenient to rewrite the previous inequality as 
\beq\label{prop:lambdaABP3}
\H^{N-1}\Big(\bigcup_{x\in K} \{\xi\in\pa B_1\cap J_Kv(x) : \, \xi\cdot\nu_\Co(x)< \lambda\}\Big)
\leq \H^{N-1}\big(\{\xi\in \pa B_1: \xi\cdot e_N< \lambda\}\big)
\eeq
as $\H^{N-1}\Big(\bigcup_{x\in K} \{\xi\in\pa B_1\cap J_Kv(x) : \, \xi\cdot\nu_\Co(x)= \lambda\}\Big)=0$.
\begin{proof}[\bf Proof of Theorem~\ref{thm:maingeo}] From the previous discussion it is enough to consider functions defined on finite subsets of $\pa\Co$. Let us fix $K= \{x_1, x_2, \dots x_n\} \subset \pa\Co$ and $v: K\to \R$. Then, keeping in mind Remark~\ref{rm:ovv}, we have that the sets $A_i$ defined as
\[
A_i := \text{int}\{\xi \in \R^{N} : \xi \in J_K v (x_i) \} 
\] 
are disjoint and 
\beq\label{vanishing0}
\cN:=\R^N\setminus \cup_{i}^nA_i \text{ satisfies } |\cN|=0\,.
\eeq
Moreover, as they are obtained as a finite intersection of open half-spaces, we have
\beq\label{vanishing}
\H^{N-1}(\pa A_i\cap \pa B_{\varrho})=0 \qquad\text{for all $i=1,\dots, n$ and all $\varrho>0$.}
\eeq
Thus, for any $ \xi \in \cup_{i}A_i$ we can define the associated normal 
\[
\nu(\xi) := \nu_\Co(x_i)\,,
\]
where $i$ is such that $ \xi \in A_i$. Recall that by the half-line property stated in Lemma~\ref{lem:v-gamma}, we have
\beq\label{halfline}
\xi\in A_i\Rightarrow \xi+t\nu(\xi)\in A_i\quad\text{ for all }t>0\,.
\eeq

For $\lambda \in (-1,1)$ denote $r= r(\lambda): = \sqrt{1 -\lambda^2}$ and define 
\[
\pa B_r^- = \{ \xi \in \pa B_r\setminus \cN : \xi \cdot \nu(\xi) < 0\} \quad \text{and} \quad \pa B_r^+ = \{ \xi \in \pa B_r\setminus \cN : \xi \cdot \nu(\xi) > 0\}
\]
Now observe that, by the half-line property, for every $\xi\in \R^N\setminus \cN$ we have
\beq\label{newone}
\nu\big(\xi+t\nu(\xi)\big) = \nu(\xi)\qquad \forall t>0\,.
\eeq
In particular, if $\xi\in \pa B_r^-$, then the symmetric point $\hat \xi:=\xi-2 (\xi\cdot \nu(\xi))\nu(\xi)\in \pa B_r^+$. Thus, $\H^{N-1}(\pa B_r^-)\leq \H^{N-1}(\pa B_r^+)$ and in turn, since $\H^{N-1}( \{ \xi \in \pa B_r\setminus \cN : \xi \cdot \nu(\xi) = 0\})=0$, we conclude that 
\beq\label{lambda=0bis}
\H^{N-1}(\pa B_r^-)\leq \frac 12\H^{N-1}(\pa B_r)\leq \H^{N-1}(\pa B_r^+)\,.
\eeq
We now distinguish the two cases $\lambda<0$ and $\lambda>0$. 

\noindent{\bf The case $\lambda< 0$.} 
Let us set
\[
S_{\lambda} := \{ \xi \in \pa B_1\setminus \cN : \xi \cdot \nu(\xi) < \lambda\}\,. 
\]
and
\[
\varphi(\lambda) :=\Ha^{N-1}\big(\{ \xi \in \pa B_1 : \xi \cdot e_N < \lambda\}\big)\,.
\]
By the discussion above it is enough to prove~\eqref{prop:lambdaABP3}, which may be rewritten as 
\[
\H^{N-1}(S_\lambda)\leq \varphi(\lambda)\,\qquad \text{ for all }\lambda \in (-1,0)\,.
\]
In fact, as mentioned in the introduction, we will show that the above inequality holds for any finite collection of open sets $(A_i)$ satisfying 
\eqref{vanishing0}, \eqref{vanishing}, and~\eqref{halfline}.

Let us fix $\eps>0$ and prove that for all $\lambda \in (-1,0)$ it holds 
\begin{equation} \label{eq:to-show-eps}
\Ha^{N-1}\big(S_{\lambda} \big) < \varphi(\lambda)+\eps\, .
\end{equation}
By continuity~\eqref{eq:to-show-eps} holds when $\lambda$ is close to $-1$. Let us assume by contradiction that there exists $\lambda <0$ such that 
\beq\label{contra}
\Ha^{N-1}\big(S_{\lambda} \big) = \varphi(\lambda)+\eps \quad\text{and}\quad \Ha^{N-1}\big(S_{\lambda'} \big) < \varphi(\lambda')+\eps\, \text{ for all }\lambda'\in [-1, \lambda)\,.
\eeq

We define the map $\Psi : S_{\lambda} \to \pa B_r^-$, 
\[
\Psi(\xi) = \xi + \big( -\sqrt{r^2 -1 + (\xi \cdot \nu(\xi))^2} -(\xi \cdot \nu(\xi)) \big) \, \nu(\xi)\,. 
\]
The geometric meaning of the map is the following: let $\sigma_\xi$ be the half-line $\{\xi+t\nu(\xi):\, t>0\}$. Then $\xi\in S_\lambda$ if and only if $\sigma_\xi$ intersects $\pa B_r$ at two points (see Figure~\ref{cdotandr}), and in this case $\Psi(\xi)$ is the intersection point closer to $\xi$. 

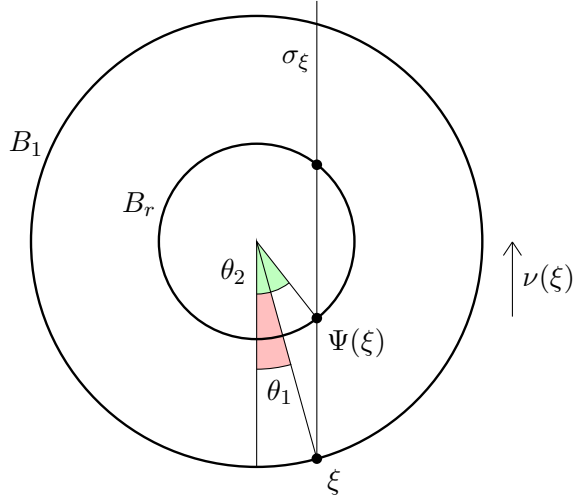
\begin{figure}[htbp]
\begin{tikzpicture}[>=>>>>]
%\draw[line width=.9pt] (0,0) circle (2);
%\draw[line width=.9pt] (0,0) circle (.6);
%\draw (-.56,-1.92) -- (1.4,2);
%\draw [->] (2.5,-1) -- (3,0);
%\draw (2.72,-.62) node[anchor=west] {$\nu(\xi)$};
%\draw (-1.7,1.3) node[anchor=east] {$B_1$};
%\draw (-.4,.5) node[anchor=east] {$B_r$};
%\fill (-.56,-1.92) circle (2pt);
%\fill (0.54,0.3) circle (2pt);
%\fill (0.1,-0.58) circle (2pt);
%\draw (1.2,-1.1) node[anchor=south east] {$\Psi(\xi)$};
%\draw (-.56,-1.82) node[anchor=north east] {$\xi$};
%\draw (1.25,1.5) node[anchor=south west] {$\sigma_\xi$};
%\end{tikzpicture}
\fill[red!25] (0,-1.7) arc(-90:-74.3:1.7) -- (0,0) -- (0,-1.7);
\fill[green!25] (0,-.7) arc (-90:-52:.7) -- (0,0) -- (0,-.7);
\fill[green!25] (0,-.7) arc (-90:-74.3:.7) -- (0,0) -- (0,-.7);
\draw[line width=.9pt] (0,0) circle (3);
\draw[line width=.9pt] (0,0) circle (1.3);
\draw [->] (3.4,-1) -- (3.4,0);
\draw (3.4,-.5) node[anchor=west] {$\nu(\xi)$};
\draw (-2.7,1.3) node[anchor=east] {$B_1$};
\draw (-1.2,.5) node[anchor=east] {$B_r$};
\fill (.8,-2.89) circle (2pt);
\draw (.8,-2.89) node[anchor=north west] {$\xi$};
\fill (.8,-1.02) circle (2pt);
\draw (.8,-1.02) node[anchor=north west] {$\Psi(\xi)$};
\fill (.8,1.02) circle (2pt);
\draw (.8,-2.89) -- (.8,3.2);
\draw (0,-3) -- (0,0) -- (.8,-2.89);
\draw (.2,2.4) node[anchor=west] {$\sigma_\xi$};
\draw (0,0) -- (.8,-1.02);
\draw (0,-1.7) arc (-90:-74.3:1.7);
\draw (0,-.7) arc (-90:-52:.7);
\draw (0,-2) node[anchor=west] {$\theta_1$};
\draw (0,-.4) node[anchor=east] {$\theta_2$};
\end{tikzpicture}
\caption{The meaning of the radius $r$ and the definition of $\Psi$. Note that $\xi\cdot \nu(\xi)< \lambda<0$ if and only if the half-line $\sigma_\xi$ intersects $B_r$. Note also that $J^{N-1}_\Psi(\xi)=\cos(\theta_1)/\cos(\theta_2)$.}
\label{cdotandr}
\end{figure}
An elementary calculation shows that indeed $|\Psi(\xi)| = r$ and that, also by~(\ref{newone}),
\[
\Psi(\xi)\cdot \nu(\Psi(\xi))=\Psi(\xi)\cdot \nu(\xi)<0
\]
for all $\xi \in S_{\lambda}$.

Note also that $\Psi$ is smooth and thus we may calculate its tangential $(N-1)$-Jacobian $J_\Psi^{N-1}$. By a tedious but straightforward computation, recalling also that $\xi\mapsto \nu(\xi)$ is locally constant, we get for all $\xi \in S_{\lambda}$
\[
J^{N-1}_\Psi(\xi) = \frac{r |\xi \cdot \nu(\xi)|}{ \sqrt{r^2 -1 + (\xi \cdot \nu(\xi))^2}}\,.
\]
For a ``geometric'' interpretation of this fomula, see Figure~\ref{cdotandr}.

We now compute $\Ha^{N-1}(\Psi(S_{\lambda}))$ using the Area formula (see for instance~\cite[Theorem~2.91]{AmbrosioFuscoPallara00}):
\[
\begin{split}
\Ha^{N-1}(\Psi(S_{\lambda})) &= \int_{S_{\lambda}} J^{N-1}_\Psi \, d\Ha^{N-1} = \int_{S_{\lambda}} \frac{r |\xi \cdot \nu(\xi)|}{ \sqrt{r^2 -1 + (\xi \cdot \nu(\xi))^2}} \, d\Ha^{N-1}\\
&=\int_0^\infty \Ha^{N-1}\Big(\{ \xi \in S_{\lambda} : \frac{r \, |\xi \cdot \nu(\xi)|}{ \sqrt{r^2 -1 + (\xi \cdot \nu(\xi))^2}} \geq t \} \Big) \, dt\,, 
\end{split}
\] 
where the last equality follows by the layer-cake formula.

Notice that $ \frac{r \, |\xi \cdot \nu(\xi)|}{ \sqrt{r^2 -1 + (\xi \cdot \nu(\xi))^2}} \geq 1$ for all $\xi \in S_{\lambda}$, while for $t >1$ 
\[
\frac{r \, |\xi \cdot \nu(\xi)|}{ \sqrt{r^2 -1 + (\xi \cdot \nu(\xi))^2}} \geq t \quad \text{is equivalent to }\quad \xi \cdot \nu(\xi) \geq -\frac{t \, \sqrt{1 -r^2}}{\sqrt{t^2 - r^2}} = : \ell(t)\,, 
\] 
where we used also that $\xi \cdot \nu(\xi)<0$ for $\xi \in S_{\lambda}$. Recalling that $r = \sqrt{1 - \lambda^2}$ we may write the latter condition as 
\[
\xi \cdot \nu(\xi) \geq \ell(t) = \frac{t\lambda}{ \sqrt{t^2 +\lambda^2 -1}}\,. 
\]
Note that $\ell(1) = -1$, $\lim_{t \to \infty}\ell(t) = \lambda$, and $\ell(\cdot)$ is increasing, so that $-1<\ell(t)<\lambda$ for all $t\in (1,+\infty)$. Therefore, we have 
\beq\label{layer}
\begin{split}
\Ha^{N-1}(\Psi(S_{\lambda})) &=\int_1^\infty \Ha^{N-1}\Big(\{ \xi \in S_{\lambda} : \xi \cdot \nu(\xi) \geq \ell(t) \} \Big) \, dt + \Ha^{N-1}(S_{\lambda}) \\
&= \int_1^\infty \Ha^{N-1}\Big(\{ \xi \in \partial B_1 : \lambda > \xi \cdot \nu(\xi) \geq \ell(t) \} \Big) \, dt + \Ha^{N-1}(S_{\lambda}) \\
&= \int_1^\infty \Big( \Ha^{N-1}(S_{\lambda}) - \Ha^{N-1}(S_{\ell(t)}) \Big) \, dt + \Ha^{N-1}(S_{\lambda})\,.
\end{split}
\eeq
By the contradiction assumption~\eqref{contra} it holds 
\[
\Ha^{N-1}(S_{\lambda}) = \varphi(\lambda)+\eps \quad \text{and} \quad \Ha^{N-1}(S_{\ell(t)}) < \varphi(\ell(t))+\eps
\]
for all $t>1$ and therefore by~\eqref{layer}
\beq\label{quasi}
\Ha^{N-1}(\Psi(S_{\lambda})) > \int_1^\infty \big( \varphi(\lambda) - \varphi(\ell(t)) \big) \, dt + \varphi(\lambda)+\eps\,. 
\eeq
A direct computation of the last integral seems to be tricky. However, we can overcome the difficulty by the following geometric argument: we pick any point $\bar x\in \pa\Co$ and consider the trivial situation $v:\{\bar x\}\to \R$. Without loss of generality we may assume that $\nu_\Co(\bar x)=e_N$. In this case, we have only one subdifferential $A_1=\R^N$, the map $\nu(\xi)$ introduced above trivialises to the constant map $ e_N$, 
the set $S_\lambda$ and the map $\Psi$ are replaced respectively by $\widetilde S_{\lambda} = \{ \xi \in \pa B_1 : \xi \cdot e_N < \lambda\}$ and 
\[
\widetilde \Psi(\xi) = \xi + \big( -\sqrt{r^2 -1 + (\xi \cdot e_N)^2} -(\xi \cdot e_N) \big) \, e_N\,. 
\]
It is immediate to verify that in this case $\widetilde \Psi(\widetilde S_\lambda)=\{\xi\in \pa B_r:\, \xi\cdot e_N<0\}$ and thus, by
applying~\eqref{layer} to $\widetilde \Psi$ and $\widetilde S_\lambda$, we get 
\[
\int_1^\infty \big( \varphi(\lambda) - \varphi(\ell(t)) \big) \, dt + \varphi(\lambda) = \frac12 \Ha^{N-1}(\partial B_r)\,. 
\]
Combining with~\eqref{quasi} yields 
\[
\Ha^{N-1}(\Psi(S_{\lambda})) >\frac{1}{2}\Ha^{N-1}(\partial B_r)+\eps
\]
which contradicts the first inequality in~\eqref{lambda=0bis} as $\Psi(S_{\lambda})\subset \pa B_r^-$. Hence, we have~\eqref{eq:to-show-eps} and the conclusion follows from the arbitrariness of $\eps$. 

\noindent {\bf The case $\lambda>0$.} The argument for this case is ``symmetric'' to the previous one. To this aim, let us set
\[
S^+_{\lambda} := \{ \xi \in \pa B_1\setminus \cN : \xi \cdot \nu(\xi) > \lambda\} 
\]
and
\[
\varphi^+(\lambda) :=\Ha^{N-1}\big(\{ \xi \in \pa B_1 : \xi \cdot e_N > \lambda\}\big)\,.
\]
Recalling~\eqref{tesi}, it is enough to prove that for any small $\eps>0$
\[
\H^{N-1}(S^+_\lambda)> \varphi^+(\lambda)-\eps\,\qquad \text{ for all }\lambda \in (0, 1)
\]
and as before, we argue by contradiction by assuming that there exists $\lambda >0$ such that 
\beq\label{contrabis}
\Ha^{N-1}\big(S^+_{\lambda} \big) = \varphi^+(\lambda)-\eps \quad\text{and}\quad \Ha^{N-1}\big(S^+_{\lambda'} \big) > \varphi^+(\lambda')-\eps\, \text{ for all }\lambda'\in (\lambda, 1)\,.
\eeq

We define the map $\Psi^+ : S^+_{\lambda} \to \pa B_r$, 
\[
\Psi^+(\xi) = \xi + \big( \sqrt{r^2 -1 + (\xi \cdot \nu(\xi))^2} -(\xi \cdot \nu(\xi)) \big) \, \nu(\xi)\,. 
\]
The geometric meaning of the map is the following: let $\sigma^-_\xi$ be the half-line $\{\xi-t\nu(\xi):\, t>0\}$. Then $\xi\in S^+_\lambda$ implies that $\sigma^-_\xi$ intersects $\pa B_r$ at two points and $\Psi^+(\xi)$ is the intersection point closer to $\xi$. This time, we claim that 
\beq\label{thistime}
\pa B^+_r\subset\Psi^+(S^+_\lambda)\,.
\eeq
Indeed, let $ \xi'\in \pa B^+_r$. Then, the half-line $\{ \xi'+t\nu(\xi'):\, t>0\}$ intersects $\pa B_1$ at one point $\widehat \xi$ such that $\widehat \xi\cdot \nu(\xi')>\lambda$. By the half-line property~\eqref{halfline}, $ \xi'$ and $\widehat \xi$ belong to the same subdifferential and thus 
$\nu( \xi')=\nu(\widehat \xi)$. Therefore, by the very definition of $\Psi^+$ we have $\Psi^+(\widehat \xi)=\xi'$ and thus~\eqref{thistime} follows. In turn, recalling~\eqref{lambda=0bis}, it follows that 
\beq\label{lambda=0tris}
\H^{N-1}\big(\Psi^+(S^+_\lambda)\big)\geq \frac 12\H^{N-1}(\pa B_r)\,.
\eeq

Arguing as before, by the Area and layer-cake formulas, observing that $\Psi^+$ is injective, we get
\[
\Ha^{N-1}(\Psi^+(S^+_{\lambda}))= \int_0^\infty \Ha^{N-1}\Big(\{ \xi \in S^+_{\lambda} : \frac{r \, \xi \cdot \nu(\xi)}{ \sqrt{r^2 -1 + (\xi \cdot \nu(\xi))^2}} \geq t \} \Big) \, dt\,.
\] 

Now for $t >1$, $ \frac{r \, \xi \cdot \nu(\xi)}{ \sqrt{r^2 -1 + (\xi \cdot \nu(\xi))^2}} \geq t $ is equivalent to 
\[
\xi \cdot \nu(\xi) \leq \ell(t):= \frac{t\lambda}{ \sqrt{t^2 +\lambda^2 -1}}\,,
\]
where $\lambda<\ell(t)<1$ for all $t\in (1,+\infty)$. Arguing as in the previous case, we may write
\[
\Ha^{N-1}(\Psi^+(S^+_{\lambda})) = \int_1^\infty \Big( \Ha^{N-1}(S^+_{\lambda}) - \Ha^{N-1}(S^+_{\ell(t)}) \Big) \, dt + \Ha^{N-1}(S^+_{\lambda})
\]
and, using the contradiction assumption~\eqref{contrabis}, we arrive at 
\[
\Ha^{N-1}(\Psi(S^+_{\lambda})) < \int_1^\infty \big( \varphi^+(\lambda) - \varphi^+(\ell(t)) \big) \, dt + \varphi^+(\lambda)-\eps= \frac{1}{2}\Ha^{N-1}(\partial B_r)-\eps\,,
\]
contradicting~\eqref{lambda=0tris}.
\end{proof}

%Let us denote by $x_1,\dots,x_n$ the points in $\Gamma$. 

\section{The capillary isoperimetric inequality outside convex sets}

In this section we finally prove Theorem~\ref{thm1}. We will need the following approximation lemmas whose delicate and technical proofs are postoponed until the final Appendix.

\begin{lemma}\label{appro}
Let $\Co\subset\R^N$ be a closed convex set with nonempty interior. Let $E\subset\R^N\setminus\Co$ be a bounded set of finite perimeter. Then there exist a sequence of smooth closed convex sets $\Co_n$ and a sequence of open sets $\Om_n\subset\R^N\setminus\Co_n$, satisfying the following properties:
\vskip 0.2cm \par\noindent
{\rm (i)}\, $\Co_n\to\Co$ in the Kuratowski sense;
\vskip 0.2cm \par\noindent
{\rm (ii)}\, $\Om_n$ is a bounded Lipschitz domain s.t. $\Sigma_n:=\pa\Om_n\setminus\Co_n$ 
is a $(N-1)$-manifold with boundary of class $C^{\infty}$;
\vskip 0.2cm \par\noindent
{\rm (iii)}\, $|\Om_n\triangle E|\to 0$ as $n\to\infty$, $\pa\Om_n\subset\big\{x:\,{\rm dist}(x,\pa E)<\frac{1}{n}\big\}\,$;
\vskip 0.2cm \par\noindent
{\rm (iv)}\, $P(\Om_n;\R^N\setminus\Co_n)\to P(E;\R^N\setminus\Co)\,$;
\vskip 0.2cm \par\noindent
{\rm (v)}\, $\H^{N-1}(\pa\Om_n\cap \Co_n)\to\H^{N-1}(\pa^*E\cap\Co)$ and $\H^{N-1}\res(\pa\Om_n\cap\Co_n)\wstar\H^{N-1}\res(\pa^*E\cap\Co)$ in the sense of measures.
\end{lemma}

\begin{remark}\label{rm:fessa}
Note that if $E\subset \R^N\setminus \Co$ is a bounded set of finite perimeter and $\Om_n$, $\Co_n$ are the approximating sets given in the previous lemma, then clearly
\[
J_{\lambda, \Co_n}(\Om_n)\to J_{\lambda, \Co}(E)\,.
\] 
\end{remark}

Next approximation lemma is needed to deal with the characterization of the equality case in Theorem~\ref{thm1}.

\begin{lemma}\label{approbis}
Let $\Co\subset\R^N$ be a closed convex set with nonempty interior. Let $\Om\subset\R^N\setminus\Co$ be a minimizer of problem~\eqref{def:min-prob} for some $m>0$. Then there exist a sequence of smooth closed convex sets $\Co_n$ and a sequence of open sets $\Om_n\subset\R^N\setminus\Co_n$, satisfying properties {\rm (i)--(v)} of Lemma~\ref{appro} with $E$ replaced by $\Om$, and in addition:
\vskip 0.2cm \par\noindent
{\rm (vi)} if $x\in\pa\Om\setminus\Co$, $\overline B_r(x)\subset \R^N\setminus \Co$, and $\overline B_r(x)\cap\Sigma_{sing}=\emptyset$, where $\Sigma_{sing}\subset\pa\Om\setminus\Co$ is the set of singular points of $\pa\Om\setminus\Co$, then $\pa\Om_n\cap \overline B_r(x)$ converge to $\pa\Om\cap \overline B_r(x)$ in $C^\infty$.
\end{lemma}

\begin{proof}[\textbf{Proof of Theorem~\ref{thm1}}] For simplicity, we consider here the case when $\Co$ has nonempty interior. The extension to the general case is described in Remark~\ref{empty}. We start by showing the inequality~\eqref{main1} if $\Co$ and $E=\Omega\subset\R^N\setminus\Co$ satisfy~\eqref{Coreg} and ~\eqref{Omreg}, respectively, and $\Om$ is connected. 

By scaling we may assume $m=|B^\lambda|$. Let $u :\overline \Omega \to \R$ be the variational solution of the Neumann boundary problem ~\eqref{eq:neumann-0} and denote its restriction to $\Gamma \subset \pa \Co$ by $u_\Gamma$. Let $\mathcal B_{u_{\Gamma}}^\lambda$ be the set defined in~\eqref{blambdau}, with $v=u_{\Gamma}$, and let $\widehat \Omega$ be the set defined in~\eqref{omegahat}. Then by Lemma~\ref{lem:u-gamma} and Theorem~\ref{thm:maingeo} we have
\[
|\nabla u(\widehat \Om)|\geq |\mathcal B_{u_{\Gamma}}^\lambda | \geq |B^\lambda|\, . 
\]
Therefore we have 
\begin{equation}\label{eq:cabre}
\begin{split}
|\Omega|=|B^\lambda|&\leq |\nabla u(\widehat\Om)| \leq \int_{\widehat\Om}\det\nabla^2u \, dx \\ &\leq \int_{\widehat\Om}\frac{(\Delta u)^N}{N^N}\, dx 
= \left( \frac{\J(\Omega )}{|\Omega| N}\right)^N |\widehat\Om|\leq \left( \frac{\J(\Omega )}{|\Omega| N}\right)^N |\Omega|\,. 
\end{split}
\end{equation}
Since
\[
J_{\lambda,\textbf{H}}= N |B^\lambda| = N|\Omega|\,, 
\]
the inequality~\eqref{main1} follows. 

We now remove the connectedness assumption, that is we consider $\Co$ and $E=\Omega\subset\R^N\setminus\Co$ satisfying~\eqref{Coreg} and ~\eqref{Omreg}. Decomposing $\Omega=\cup_{i=1}^n\Om_i$, where $\Om_i$ are the connected components, and setting $m_i:=|\Om_i|$, we have
\beq\label{connect}
\J(\Om)=\sum_{i=1}^n \J(\Om_i)\geq \sum_{i=1}^n J_{\lambda,\textbf{H}}(B^\lambda[m_i])>J_{\lambda,\textbf{H}}(B^\lambda[m])\,, 
\eeq
where the last inequality follows from strict concavity of the map $m\mapsto J_{\lambda,\textbf{H}}(B^\lambda[m])$. The general case of a bounded set $E$ of finite perimeter now follows by approximation, using Lemma~\ref{appro} and Remark~\ref{rm:fessa}. And finally, the case of a possible unbounded set follows by approximating $E$ in $L^1$ with a sequence $E_n\subset \R^N\setminus \Co$ such that $P(E_n)\to P(E)$, since this implies that $P(E_n;\R^N\setminus \Co)$ and $\H^{N-1}(\partial^* E_n\cap \Co)$ converge to $P(E;\R^N\setminus \Co)$ and $\H^{N-1}(\partial^* E\cap \Co)$, respectively.

We now analyse the case of equality. Assume that $E$ is a set of finite perimeter with $|E|=m=|B^\lambda|$ for which equality in~\eqref{main1} holds. In particular, $E$ is a minimizer of the isoperimetric problem~\eqref{def:min-prob} and therefore it coincides (up to negligible sets) with an open set $\Om$, see for instance~\cite{FFLM22}. Moreover, by the same argument as in~\eqref{connect}, we know that $\Om$ is connected. 

Let $\Om_n$ and $\Co_n$ be the two approximating sequences provided by Lemma~\ref{approbis}, and denote the connected components of $\Om_n$ as $\Om^i_n$, with $i=1,\, 2,\, \dots\, ,\, K_n$, where $|\Om^1_n|\geq |\Om^i_n|$ for every $i$. Let now $u_{n,i}$ be the variational solution of 
\beq\label{senew}
\begin{cases}
&\Delta u_{n,i} = \frac{J_{\lambda,\Co_n}(\Om^i_n)}{|\Omega^i_n|} \quad \text{in } \, \Omega^i_n \\
&\pa_\nu u_{n,i} = 1 \quad \text{on } \, \Sigma^i_n\\ 
&\pa_\nu u_{n,i} = -\lambda \quad \text{on } \, \Gamma^i_n\,, 
\end{cases}
\eeq
where $\Sigma^i_n=\pa\Om^i_n\setminus \Co_n$, $\Gamma^i_n:=\pa\Om^i_n\cap\Co_n$. Arguing as in the first part of the proof (see~\eqref{eq:cabre}), for every $1\leq i \leq K_n$ we have
\beq\label{eq:cabrebis}
\begin{split}
|B^{\lambda}|&\leq |\mathcal B^\lambda_{u_{n,i}}|\leq |\nabla u_{n,i}(\widehat\Om^i_n)|\leq \int_{\widehat\Om^i_n}\det\nabla^2 u_{n,i}\, dx\\
&\leq \int_{\widehat\Om^i_n}\frac{(\Delta u_{n,i})^N}{N^N}\, dx\leq \frac{(J_{\lambda,\Co_n}(\Om^i_n))^N}{N^N|\Om^i_n|^N }|\widehat\Om^i_n|
\leq \frac{(J_{\lambda,\Co_n}(\Om^i_n))^N}{N^N|\Om^i_n|^N }|\Om^i_n|\,,
\end{split}
\eeq
where 
\[
\widehat\Om^i_n:=\big\{x\in \Om^i_n: J_{{\overline \Om}^i_n}u_{n,i}(x)\neq\emptyset\text{ and }\nabla u_{n,i}(x)\in B_1\big\}\,.
\]
Summing up the inequality we deduce
\beq\label{vesa1}
J_{\lambda,\Co_n}(\Om_n) = \sum_{i=1}^{K_n} J_{\lambda,\Co_n}(\Om^i_n) \geq N |B^\lambda|^{\frac 1N} \sum_{i=1}^{K_n} |\Omega^i_n|^{\frac{N-1}N}
\geq N |B^\lambda|^{\frac 1N} |\Omega_n|^{\frac{N-1}N}\,.
\eeq
Note that by properties (iii), (iv) and (v) of Lemma~\ref{appro} we have that $|\Om_n|\to |\Om|=|B^\lambda|$ and
\beq\label{qeq}
J_{\lambda,\Co_n}(\Om_n)\to J_{\lambda,\Co}(\Om)=J_{\lambda,\Co}(B^\lambda)=N|B^\lambda|\,,
\eeq
and combining this with the last inequality in~(\ref{vesa1}) we deduce that $|\Om^1_n|\to |B^\lambda|$, while $\sum_{i=2}^{K_n} |\Om^i_n| \to 0$. Therefore, up to replacing $\Om_n$ with $\Om^1_n$, we may from now on assume that $\Om_n$ is connected, and we simply write $u_n$ and $\widehat\Omega_n$ in place of $u_{n,1}$ and $\widehat\Om_{n,1}$.\par

Let us observe that (up to additive constants), we may assume that the functions $u_n$ vanish at points $x_n$ of $\widehat\Om_n$ uniformly far from $\partial \Omega_n$. Thus, we have
\[
u_n(y)\geq \nabla u_n(x_n)\cdot (y-x_n)\geq -\mathrm{diam\,}(\Om_n)\quad \text{for all }y\in \Om_n\,,
\]
where we used the fact that $ |\nabla u_n(x_n)|<1$.
Thus the $u_n$'s are uniformly bounded from below. Since they solve the equation $\Delta u_n=c_n$, with $c_n$ uniformly bounded, it follows from a standard Harnack inequality that 
\[
\sup_n \|u_n\|_{L^\infty(\Om')}<+\infty \quad\text{for all } \Om'\subset\!\subset \Om\,.
\]
In turn, recalling also~\eqref{qeq} and by standard elliptic regularity, we may assume that there exists $u\in C^\infty(\Om)$ such that up to extracting a (non relabelled) subsequence
\[
\Delta u= N\, \text{ in }\Om \quad \text{ and }\quad u_n\to u \in C^\infty(\overline \Om')\, \text{ for all } \Om'\subset\!\subset \Om\,.
\]
Note now that by~\eqref{qeq}, the inequalities in~\eqref{eq:cabrebis} (with $K_n=1$) become equalities in the limit for $u$. In particular, $|\widehat\Om_n|\to |\Om|=|B^\lambda|$ and since $|\Om_n\triangle \Om|\to 0$, we have (up to a non relabelled subsequence) 
\beq\label{senzanome}
\Chi{\widehat\Om_n}\to \Chi{\Om}\text{ almost everywhere.}
\eeq
Thus, by Fatou Lemma applied to the positive functions $\frac{(\Delta u_n)^N}{N^N}-\det\nabla^2 u_n$, we may pass to the limit in~\eqref{eq:cabrebis} to conclude that 
\[
\int_{\Om}\det\nabla^2u\, dx= \int_{\Om}\frac{(\Delta u)^N}{N^N}\, dx
\]
and, in turn, since $\det\nabla^2u \leq \big(\frac{\Delta u}{N}\big)^N=1$,
\[
\det\nabla^2u= 1 \quad\text{in $\Om$}\,.
\]
The above equality in the arithmetic-geometric mean inequality implies that all the eigenvalues of $\nabla^2u$ are equal to 1 in $\Om$. Thus, $\nabla^2u=I$ in $ \Om$ and in turn by the connectedness of $\Om$, there exist $x_0\in \R^N$ and $b\in \R$ such that
\beq\label{u}
u(x)=\frac12|x-x_0|^2+b\quad\text{for all }x\in \Om\,.
\eeq
Note also that since $|\nabla u_n|<1$ in $\widehat\Om_n$ and recalling~\eqref{senzanome}, we have also that 
$|\nabla u|\leq 1$ in $\Om$ and thus~\eqref{u} implies $\Om\subset B_1(x_0)$.

% In particular, since also $|\nabla u|\leq 1$, we have
% \beq\label{prima}
% \Om\subset B_1(x_0).
% \eeq

We now study the boundary conditions satisfied by $u$. To this aim, let $\overline B_{2r}(x)\subset \R^N\setminus \Co$, with $x\in \pa\Om$ and $\overline B_{2r}(x)\cap \Sigma_{sing}=\emptyset$, where $\Sigma_{sing}$ is the possibly empty singular set of $\pa \Om\setminus \Co$. By (vi) of Lemma~\ref{approbis} we have that $\pa\Om_n\cap \overline B_{2r}(x)$ converge in $C^\infty$ to $\pa \Om\cap \overline B_{2r}(x)$. Since by~(\ref{senew})
\beq\label{wform}
\begin{split}
\int_{B_{2r}(x)\cap \Om_n }\nabla u_n\cdot \nabla\varphi=-c_n &\int_{B_{2r}(x)\cap \Om_n}\varphi\, dx\\
&+\int_{\pa \Om_n\cap B_{2r}(x)}\varphi\, d\H^{N-1}\quad\text{for all }\varphi\in H^1_0(B_{2r}(x))\,,
\end{split}
\eeq
and since $\pa \Om_n\cap B_{2r}(x)$ are uniformly Lipschitz boundaries, by a standard Harnack Inequality up to boundary we have that, up to possibly replacing $u_n$ by $\tilde u_n:=u_n+d_n$, $d_n\in \R$, we have 
\[
\sup_n\|\tilde u_n\|_{L^\infty(B_{3r/2}(x)\cap \Om_n)}<+\infty\,.
\]
In turn, by a Caccioppoli Inequality argument and exploiting that Trace Theorem holds on $\pa \Om_n\cap B_{\frac32r(x)}$ with uniform constants, we deduce 
\[
\sup_n\|\tilde u_n\|_{H^1(\Om_n\cap B_r(x))}<+\infty\,.
\] 
Thus we may extend each $\tilde u_n$ to the whole $B_r(x)$, in such a way that 
\[
\sup_n\|\tilde u_n\|_{H^1(B_r(x))}<+\infty\,.
\] 
Hence, up to a not relabelled subsequence, $\tilde u_n\wto \tilde u$ weakly in $H^1(B_r(x))$, with $\nabla \tilde u=\nabla u$ in $\Om\cap B_r(x)$. Therefore, we can pass to the limit in~\eqref{wform} to get
\[
\int_{B_r(x)\cap \Om }\nabla u\cdot \nabla\varphi=-N\int_{B_r(x)\cap \Om}\varphi\, dx+\int_{\pa \Om\cap B_r(x)}\varphi\, d\H^{N-1}\quad\text{for all }\varphi\in C^\infty_c(B_r(x))\,,
\]
which yields $\pa_\nu u=1$ on $\pa\Om\cap B_r(x)$ and thus on $\pa\Om\setminus (\Co\cup\Sigma_{sing})$ by the arbitrariness of $B_r(x)$. In turn, since $1=\pa_\nu u(x)=(x-x_0)\cdot \nu_{\Om}(x)$ for all $x\in\pa\Om\setminus (\Co\cup\Sigma_{sing})$ and recalling that $\Om\subset B_1(x_0)$, we have necessarily $\pa\Om\setminus (\Co\cup\Sigma_{sing})\subset\pa B_1(x_0)$. Hence, $\Sigma_{sing}=\emptyset$ and $\pa\Om\setminus \Co\subset\pa B_1(x_0)$.
% Combining this piece of information with~\eqref{bellau} and the fact that $|\nabla u|\leq 1$ in $\Om$, we infer 
% \[
%\pa\Om\setminus \Co\subset \pa B_1(x_0).
% \]

Note now that, since $\Co_n\to \Co$ in the Kuratowski sense, it follows that the boundaries $\pa \Co_n$ are locally equi-Lipschitz. Consequently, for every ball $B_r(x)$ such that $\overline B_r(x)\cap \pa \Om=(\Gamma\setminus \gamma)\cap \overline B_r(x)$, we eventually obtain $\overline B_r(x)\cap \pa \Om_n=(\Gamma_n\setminus \gamma_n)\cap \overline B_r(x)$, allowing us to extend each $u_n$ to functions $\tilde u_n\in H^1(B_r)$ having uniformly bounded $H^1$-norms. Hence, up to extracting a subsequence (not explicitly relabeled), we may assume $\tilde u_n\rightharpoonup \tilde u$ weakly in $H^1(B_r(x))$, where $\tilde u=u$ within $\Om\cap B_r(x)$. Proceeding similarly to before, we deduce that $\pa_\nu u=-\lambda$ almost everywhere on $\Gamma\setminus \gamma$.

Take now a point $\hat x\in \Gamma\setminus \gamma$ and introduce the half-space
\[
H:=\{y\in \R^N:\, (y-x_0)\cdot\nu_{\Co}(\hat x)>\lambda\}\,.
\]
Considering that $-\lambda= \pa_\nu u(\hat x)
=(\hat x-x_0)\cdot\nu_{\Om}(\hat x)
=-(\hat x-x_0)\cdot\nu_{\Co}(\hat x)$, we infer that $\Co$ must be contained within $\R^N\setminus H$, and that the boundary $\pa H$ is tangent to $\Co$ exactly at the point $\hat x$. Additionally, for any non-tangential direction $v\in \S^{N-1}$ satisfying $v\cdot \nu_{\Co}(\hat x)>0$, all points along the half-line $\hat x+tv$, $t>0$, within the ball $B_1(x_0)$ lie entirely within $\Om$. Indeed, otherwise this half-line would intersect $\pa\Om$ at a point belonging to $B_1\setminus \Co$, leading to a contradiction to the fact that $\pa\Om\setminus \Co\subset\pa B_1(x_0)$. Thus, we conclude $B_1(x_0)\cap H\subseteq \Om$. Finally, as the set $B_1(x_0)\cap H$ is a spherical cap isometric to $B^\lambda$ and since $|\Om|=|B^\lambda|$, we deduce that $\Om=B_1(x_0)\cap H$, concluding the proof of the theorem.
\end{proof}

\begin{remark}\label{empty}
Let us here describe how the proof of Theorem~\ref{thm1} can be extended to the case of a convex set $\Co$ with empty interior. First of all, we point out that in this case the capillary energy must be defined as follows
\beq\label{empty1}
\J(E):=P(E; \R^N \setminus \Co) - \lambda\int_{\Co}\big(\mathrm{Tr}^+(\Chi{E})+\mathrm{Tr}^-(\Chi{E})\big)\,d\H^{N-1}\,,
\eeq
where $\mathrm{Tr}^\pm(\Chi{E})$ denote the traces of the characteristic function $\Chi{E}$ on both sides of $\Co$, see for instance~\cite[Theorem 3.77]{AmbrosioFuscoPallara00}. Then, inequality~(\ref{main1}) immediately follows by approximating $\Co$ with convex sets with nonempty interior. Concerning the equality case, it can be obtained as in the proof for the nonempty interior case, since both Lemma~\ref{appro} and~\ref{approbis} extend to the empty interior case. For Lemma~\ref{appro}, this is readily obtained with a further approximation of $\Co$. A direct argument to extend Lemma~\ref{approbis} is to argue exactly as in the proof presented in the Appendix, replacing $\Co$ by its $\eps$-neighborhood $\Co_\eps$, and $\Omega$ by $\Omega_\eps=\Omega\setminus \Co_\eps$ where $\Omega$ is a solution of problem~(\ref{def:min-prob}) in $\R^N\setminus\Co$.
%(and $\Omega_\eps$ is not necessarily close to a solution of~(\ref{def:min-prob}) in $\R^N\setminus \Co_\eps$).
\end{remark}

\section{Appendix}
In this section we give the proofs of the two approximation lemmas that were used to prove the main theorem.

\begin{proof}[\bf Proof of Lemma~\ref{appro}]
Let $B_R$ be a ball such that $E\subset\!\subset B_R$ and assume without loss of generality that the interior of $\Co$ contains the origin. Moreover, we can assume without loss of generality that $\H^{N-1}(\pa^* E \cap \Co)>0$, since otherwise the claim is obvious.\par
Given $\sigma>0$ we begin by constructing a sequence of smooth convex sets $\Co_{\sigma}^k\subset\R^N$, with $\Co\subset \Co_{\sigma}^k$, converging to $\Co_\sigma:=(1+\sigma)\Co$ in the Kuratowski sense as $k\to\infty$ and such that $(1+\sigma)\Co\cap B_R\subset\Co_{\sigma}^k\cap B_R$. Up to slightly dilating $\Co_{\sigma}^k$ if needed, we may always assume that 
\beq\label{appro1}
\H^{N-1}(\pa^* E\cap\pa\Co_{\sigma}^k)=0 \qquad \text{for all } \, k,\sigma\,.
\eeq
We consider the signed distance function ${\rm sd}_{\Co_{\sigma}^k}(x)$ from $\pa\Co_{\sigma}^k$, which is a $C^\infty$ function in $O_{\sigma}^k=\{x:\,{\rm sd}_{\Co_{\sigma}^k}(x)>-\eta_{\sigma}^k\}$ for some $\eta_{\sigma}^k>0$. Consider the smooth convex sets $\Co_{\sigma,s }^k:=\{x:\,{\rm sd}_{\Co_{\sigma}^k(x)}\leq s\}$ for $s>-\eta_{\sigma}^k$. 

To approximate $E$ we first extend $\Chi{E}_{\big|_{\R^N\setminus\Co}}$ to a function $u\in BV(\R^N)$, with compact support, such that $|Du|(\pa\Co)=0$, $0\leq u\leq1$, see~\cite[Proposition~3.21]{AmbrosioFuscoPallara00} or~\cite[Theorem~13.8]{Maggi12}. Note that for all $t\in(0,1)$, $\{u>t\}\setminus\Co=E$. For any $\eps>0, t\in(0,1)$ we set $U_{\eps,t}=\{x:\,u_\eps(x)>t\}$, where $u_\eps=\varrho_\eps*u$, for a standard mollifier $\varrho_\eps$. 
Note that for a.e. $t\in(0,1)$ there exists a sequence $\eps_n$ converging to zero such that 
\begin{equation}\label{appro2}
\begin{split}
& \lim_{n\to\infty}|U_{\eps_n,t}\triangle\{u>t\}|=0, \quad \lim_{n\to\infty}P(U_{\eps_n,t})=P(\{u>t\})\,, \\
&\pa U_{\eps_n,t}\subset\Big\{x:\,{\rm dist}(x,\pa\{u>t\})<\frac{1}{n}\Big\}\,,
\end{split}
\end{equation}
see~\cite[Theorem 3.42]{AmbrosioFuscoPallara00}. 
\par
Consider now the $C^\infty$ map $x\mapsto ({\rm sd}_{\Co_{\sigma}^k}(x),u_\eps(x))$ defined for all $x\in O_{\sigma}^k$. By Sard's theorem we have that 
\[
{\rm rank}\bigg(\!\!
\begin{array}{c}
\nabla {\rm sd}_{\Co_{\sigma}^k}(x) \\
\nabla u_{\eps_n}(x)
\end{array}
\!\!
\bigg)=2 \quad\text{on $\big\{x:\,{\rm sd}_{\Co_{\sigma}^k}(x)=s,\,\,u_{\eps_n}(x)=t\big\}$ for a.e. $(s,t)\in(0,\infty)\times (0,1)\,$.}
\]
Keeping in mind that $\H^{N-1}(\pa^* E \cap \Co)>0$, and provided that $\sigma$ is small enough, a simple argument shows that there exists $\delta>0$ such that for any $t\in(0,1)$ and $s\in(0,\delta)$ the intersection $\{{\rm sd}_{\Co_{\sigma}^k}=s\}\cap\{\, u_{\eps_n}=t\}$ is not empty for all $ n$ sufficiently large.
Hence, we may fix from now on $t\in (0,1)$ satisfying~\eqref{appro2} and such that for a.e. $s\in (0, \delta)$ the above rank condition holds for all $n$. Therefore for a.e. $s\in (0, \delta)$ the open set $\Om_{\sigma,\eps_n,s}^k=U_{\eps_n,t}\setminus\Co_{\sigma,s}^k$ is a Lipschitz domain such that $\pa\Om_{\sigma,\eps_n,s}^k\setminus \Co_{\sigma,s}^k$ is a $C^\infty$ manifold with boundary. Note that for any $\sigma$ and $k$ we have that for a.e. $s$, $\H^{N-1}(\pa^* E\cap\pa\Co_{\sigma,s}^k)=0$. Therefore for all such $s$, we have
\beq\label{appro2.5}
\begin{split}
\lim_{n\to\infty}P(\Om_{\sigma,\eps_n,s}^k;\R^N\setminus\Co_{\sigma,s}^k)
&=\lim_{n\to\infty}P(U_{\eps_n,t};\R^N\setminus\Co_{\sigma,s}^k)\\
&=P(E;\R^N\setminus\Co_{\sigma,s}^k)=P(E\setminus\Co_{\sigma,s}^k;\R^N\setminus\Co_{\sigma,s}^k)\,.
\end{split}
\eeq
From the above convergence and the continuity of the trace operator for $BV$ functions, see~\cite[Theorem~3.88]{AmbrosioFuscoPallara00} we have that
\beq\label{appro2.6}
\lim_{n\to\infty}\H^{N-1}(\pa\Om_{\sigma,\eps_n,s}^k\cap\pa\Co_{\sigma,s}^k)=\H^{N-1}(\pa^*(E\setminus\Co_{\sigma,s}^k)\cap\pa\Co_{\sigma,s}^k)=\H^{N-1}(E\cap\pa\Co_{\sigma,s}^k)\,,
\eeq
where the last equality follows from the fact that $\H^{N-1}(\pa^* E\cap\pa\Co_{\sigma,s}^k)=0$. Observe now that, since $\Co_{\sigma,s}^k$ converge to $\Co_{\sigma}^k$ in the Kuratowski sense, as $s\to 0$, we have in particular that $\H^{N-1}\res\pa\Co_{\sigma,s}^k\wstar\H^{N-1}\res\pa\Co_{\sigma}^k$, see for instance~\cite[Remark~2.2]{FFLM22}. Therefore, thanks to~\eqref{appro1} we conclude that $\H^{N-1}(E\cap\pa\Co_{\sigma,s}^k)\to\H^{N-1}(E\cap\pa\Co_{\sigma}^k)=\H^{N-1}(\pa^*(E\setminus\Co_{\sigma}^k)\cap\pa\Co_{\sigma}^k)$. Thus we have
\beq\label{appro3}
\begin{split}
&\lim_{s\to 0}P(E\setminus\Co_{\sigma,s}^k;\R^N\setminus\Co_{\sigma,s}^k)=P(E\setminus\Co_{\sigma}^k;\R^N\setminus\Co_{\sigma}^k)\\
&\lim_{s\to 0}\H^{N-1}(\pa^*(E\setminus\Co_{\sigma,s}^k)\cap\pa\Co_{\sigma,s}^k)=\H^{N-1}(\pa^*(E\setminus\Co_{\sigma}^k)\cap\pa\Co_{\sigma}^k)\,.
\end{split}
\eeq
By a similar argument, if $\sigma>0$ is such that $\H^{N-1}(\pa^* E\cap\pa(1+\sigma)\Co)=0$, we have
\beq\label{appro4}
\begin{split}
&\lim_{k\to\infty}P(E\setminus\Co_{\sigma}^k;\R^N\setminus\Co_{\sigma}^k)=P(E\setminus\Co_\sigma;\R^N\setminus\Co_\sigma),\\
&\lim_{k\to\infty}\H^{N-1}(\pa^*(E\setminus\Co_{\sigma}^k)\cap\pa\Co_{\sigma}^k)=\H^{N-1}(\pa^*(E\setminus\Co_\sigma)\cap\pa\Co_\sigma)\,.
\end{split}
\eeq
Finally, we note that by monotone convergence
\beq\label{appro5}
\lim_{\sigma\to 0}P(E\setminus\Co_\sigma;\R^N\setminus\Co_\sigma)=\lim_{\sigma\to 0}P(E;\R^N\setminus\Co_\sigma)=P(E;\R^N\setminus\Co)\,.
\eeq
By scaling, this is equivalent to say that 
\[
\lim_{\sigma\to 0}P\big(\big((1+\sigma)^{-1}E\big)\setminus\Co;\R^N\setminus\Co\big)=P(E;\R^N\setminus\Co)\,.
\]
Therefore, the trace Theorem again implies that 
\beq\label{appro6}
\begin{split}
&\lim_{\sigma \to 0}\H^{N-1}(\pa^*(E\setminus\Co_\sigma)\cap\pa\Co_\sigma)\\
&=\lim_{\sigma\to 0}(1+\sigma)^{N-1}\H^{N-1}\big(\pa^*\big((1+\sigma)^{-1} E\big)\setminus\Co\big)\cap\pa\Co)
=\H^{N-1}(\pa^* E\cap \Co)\,.
\end{split}
\eeq
From this equality, together with~\eqref{appro2.5}-\eqref{appro5} we conclude, by a diagonal argument, that there exist sequences $s_n\to 0^+$, $k_n\to\infty$ and $\sigma_n\to 0^+$ such that, setting $\Om_n=\Om_{\sigma_n,\eps_n,s_n}^{k_n}$, $\Co_n=\Co^{k_n}_{\sigma_n,s_n}$, (i)--(iv) and the first part of~(v) hold. Concerning the second part of~(v), namely, the weak* convergence of $\H^{N-1}\res(\pa\Om_n\cap\Co_n)$ to $\H^{N-1}\res(\pa^*E\cap\Co)$, it easily follows by observing that, from~(\ref{appro2.5}) and the continuity of the trace operator, for every $\varphi$ in a countable dense subset of ${\rm C}_c(\R^N)$ we can generalise~(\ref{appro2.6}) as
\[
\lim_{n\to\infty} \int_{\pa\Om_{\sigma,\eps_n,s}^k\cap\pa\Co_{\sigma,s}^k} \varphi(x)\,d\H^{N-1}(x)
= \int_{E\cap\pa\Co_{\sigma,s}^k} \varphi(x)\,d\H^{N-1}(x)\,,
\]
and in the same way we get the analogous generalisations of~(\ref{appro3})--(\ref{appro6}). Thus, we conclude again by a diagonal argument.
%If $E$ is now a general set of finite perimeter, we argue as follows. Let $\{r_j\}$ be a diverging sequence such that $\H^{N-1}(\partial^* E \cap \partial B_{r_j})=0$ for every $j$, and call $E_j=E\cap B_{r_j}$, so that $|E_j\triangle E|\to 0$ and $P(E_j; \R^N\setminus \Co)\to P(E; \R^N\setminus \Co)$. For each $j$, we define the sequence $\Om_{n,j}$, satisfying~(i)--(v) with respect to $E_j$, by means of suitable sequences $\sigma_n,\, \eps_n,\, s_n,\, k_n$ (in principle depending on $j$) chosen as above. Up to another diagonal argument, we can fix such sequences $\sigma_n,\, \eps_n,\, s_n,\, k_n$ independently of $j$, so that by construction, for every $n\in\N$, in every compact $K\subset\subset\R^N$ the sequence $\Omega_{n,j}\cap K$ is constant for $j\gg 1$. This obviously defines a set $\Omega_n$ in such a way that the sequence $\{\Omega_n\}$ satisfies~(i)--(v).
\end{proof}

We now prove the second approximation lemma. 

\begin{proof}[\bf Proof of Lemma~\ref{approbis}]

We start by observing that by the classical regularity for volume-constrained perimeter minimizers $\pa\Om\setminus \Co$ is smooth outside a (relatively closed) singular set $\Sigma_{sing}$ of Hausdorff dimension at most $N-8$. Moreover, since $\Om$ is a $\Lambda$-minimiser of the capillary energy (see for instance~\cite{FFLM22}), it can be shown that $\Om$ is open and satisfies the following perimeter density estimates: there exist $c_0$, $r_0>0$ such that for all $x\in \pa \Om$ and for all $r\in (0, r_0)$
\beq\label{approbis1}
\H^{N-1}(\pa\Om\cap B_r(x))\ge c_0 r^{N-1}\,.
\eeq
Note that this estimate in particular implies that $\Omega$ is bounded, so we can take a ball $B_R$ such that $\Om\subset\!\subset B_R$; moreover, we assume without loss of generality that the interior of $\Co$ contains the origin of $\R^N$.
\par
Let $\Co_{\sigma}^k\subset\R^N$, $\sigma>0$ and $k\in\N$, be a family of smooth convex sets extactly as in the proof of Lemma~\ref{appro}, satisfying in particular~\eqref{appro1}.

We consider the signed distance function ${\rm sd}_{\Co_{\sigma}^k}(x)$ from $\pa\Co_{\sigma}^k$, which is a $C^\infty$ function in $O_{\sigma}^k=\{x:\,{\rm sd}_{\Co_{\sigma}^k}(x)>-\eta_{\sigma}^k\}$ for some $\eta_{\sigma}^k>0$. Consider the smooth convex sets $\Co_{\sigma,s }^k:=\{x:\,{\rm sd}_{\Co_{\sigma}^k(x)}\leq s\}$ for $s>-\eta_{\sigma}^k$.

Note that if $\Om$ is an open set of finite perimeter satisfying~\eqref{approbis1}, then by~\cite[Theorem~5]{ACV08} we have that the outer $(N-1)$-dimensional Minkowski content of $\pa \Om$ coincides with $P(\Om)$; that is, 
\[
P(\Om)=\lim_{\varrho\to 0^+}\frac{\mathcal L^N\big((\Om)_\varrho\setminus \Om\big)}{\varrho}\,.
\]
Thus, by the coarea formula, the above equality may be rewritten as
\beq\label{approbis2}
P(\Om)=\lim_{\varrho\to 0^+}\frac1\varrho\int_0^\varrho P\big(\{\sd_\Om<t\}\big)\, dt\,,
\eeq
where $\sd_\Om$ stands for the signed distance function from $\pa \Om$. 
For $\eps>0$ set $\sd_\Om^\eps:=\rho_\eps * \sd_\Om$, where $\rho_\eps$ is a standard mollifier, and note that by the Dominated Convergence Theorem
\[
\int_{\{0<\sd_\Om^\eps<\rho\}}|\nabla \sd_\Om^\eps|\, dx\to \mathcal L^N((\Om)_\varrho\setminus \Om)
\]
as $\eps\to 0^+$. Therefore, by~\eqref{approbis2} and again the Coarea Formula, we have for any sequence $\eps_m\to 0^+$
\[
\lim_{n\to\infty}\lim_{m\to\infty}\int_{0}^1P\Big(\Big\{\sd_\Om^{\eps_m}<\frac{t}n\Big\}\Big)\,dt=P(\Om)\,.
\]
In turn, by Fatou's Lemma, the lower semicontinuity of the perimeter and by Sard's Theorem, we obtain that for almost every $t\in (0,1)$ we have
\beq\label{approbis2.25}
\Big\{\sd_\Om^{\eps_m}<\frac{t}n\Big\}\text{ is smooth }\quad\text{and}\quad\liminf_{n\to\infty}\liminf_{m\to\infty}P\Big(\Big\{\sd_\Om^{\eps_m}<\frac{t}n\Big\}\Big)=P(\Om)\,.
\eeq
For any $m,n$ we consider the $C^\infty$ map $x\mapsto ({\rm sd}_{\Co_{\sigma}^k}(x), n\,\sd^{\eps_m}_\Om(x))$ defined for all $x\in O_{\sigma}^k$. By Sard's theorem we have that 
\[
{\rm rank}\bigg(\!\!
\begin{array}{c}
\nabla {\rm sd}_{\Co_{\sigma}^k}(x) \\
n \nabla \sd^{\eps_m}_\Om(x)
\end{array}
\!\!
\bigg)=2 ~\text{on $\big\{x:\,{\rm sd}_{\Co_{\sigma}^k}(x)=s,\,\,n\, \sd^{\eps_m}_\Om(x)=t\big\}$ for a.e. $(s,t)\in(0,\infty)\times (0,1)\,$.}
\]
Since by the minimality of $\Om$, we have $\H^{N-1}(\pa\Omega \cap \Co)>0$, a simple argument shows that there exists $\delta>0$ such that for any $t\in(0,1)$ and $s\in(0,\delta)$ the intersection $\{{\rm sd}_{\Co_{\sigma}^k}(x)=s\}\cap\{n\, \sd^{\eps_m}_\Om(x)=t\}\not=\emptyset$ for all $m, n$ sufficiently large.
Hence we can fix $t\in (0,1)$ satisfying 
\eqref{approbis2.25} such that the above rank condition holds for all $s\in (0,\delta)\setminus \mathcal Z$ and for all $m,n$ sufficiently large, where $ \mathcal Z$ is a negligible set. Note that possibly replacing $ \mathcal Z$ with a larger (not renamed) neglibile set we may also assume that 
\beq\label{approbis2.5}
\H^{N-1}(\pa\Om\cap\pa\Co_{\sigma,s}^k)=0\quad\text{for all }s\in (0,\delta)\setminus \mathcal Z\,.
\eeq
At this point, we may also extract subsequences $m_j, n_j$ such that 
\beq\label{approbis3}
\lim_{j\to\infty}P\Big(\Big\{\sd_\Om^{\eps_{m_j}}<\frac{t}{n_j}\Big\}\Big)=P(\Om)\,.
\eeq
For any $s\in (0,\delta)\setminus \mathcal Z$ and $j$ sufficiently large, we can define
\[
\Om^k_{\sigma, j, s}:=\Big\{\sd^{\eps_{m_j}}_\Om<\frac{t}{n_j}\Big\}\setminus \Co^k_{\sigma, s}\,.
\]
Note that~\eqref{approbis2.5} and~\eqref{approbis3} imply that for all $s\in (0,\delta)\setminus \mathcal Z$, we have
\beq\label{appro2.5bis}
\lim_{j\to\infty}P(\Om_{\sigma,j,s}^k;\R^N\setminus\Co_{\sigma,s}^k)=P(\Om;\R^N\setminus\Co_{\sigma,s}^k)=P(\Om\setminus\Co_{\sigma,s}^k;\R^N\setminus\Co_{\sigma,s}^k)\,.
\eeq
From the above convergence and the continuity of the trace operator we have that
\beq\label{appro2.6bis}
\lim_{j\to\infty}\H^{N-1}(\pa\Om_{\sigma,j,s}^k\cap\pa\Co_{\sigma,s}^k))=\H^{N-1}(\pa(\Om\setminus\Co_{\sigma,s}^k)\cap\pa\Co_{\sigma,s}^k)=\H^{N-1}(\Om\cap\pa\Co_{\sigma,s}^k)\,,
\eeq
where the last equality follows from~(\ref{approbis2.5}). 

We can argue exactly as in the second part of the proof of Lemma~\ref{appro}, to infer that~\eqref{appro3}--\eqref{appro6} hold. 
From these equalities, together with ~\eqref{appro2.5bis} and~\eqref{appro2.6bis}, by a diagonal argument, that there exist sequences $k_j\to\infty$, $s_j, \sigma_j\to 0$, with $s_j\in (0,\delta)\setminus \mathcal Z$, such that $\Om_j:=\Om^{k_j}_{\sigma_j, j, s_j}$ and $\Co_j:=\Co^{k_j}_{\sigma_j, s_j}$ satisfy (i)--(v) with $E$ replaced by $\Om$.

Finally, property (vi) follows by observing that if $x\in \pa \Om\setminus \Sigma_{sing}$ and $\sd_\Om$ is smooth in $\overline B_{2\rho}(x)$, then $\sd^{\eps_{m_j}}_\Om\to \sd_\Om$ in $C^{\infty}(\overline B_\rho(x))$, which in turn easily implies~(vi).
\end{proof}

\section*{Aknowledgments} 
\noindent N.~F., M.~M. and A.~P. were supported by PRIN Project 2022E9CF89, “Geometric Evolution Problems and Shape Optimization (GEPSO)”, PNRR Italia Domani, financed by European Union via the Program NextGenerationEU. N.~F., M.~M. and A.~P. are members of the Gruppo Nazionale per l’Analisi
Matematica, la Probabilit\`a e le loro Applicazioni (GNAMPA), which is part of the Istituto
Nazionale di Alta Matematica (INdAM). V.~J.~was supported by the Academy of Finland grant 314227. 

\bibliographystyle{siam}
\bibliography{isoconvex}

\begin{thebibliography}{10}

\bibitem{ACV08}
{\sc L.~Ambrosio, A.~Colesanti, and E.~Villa}, {\em Outer {M}inkowski content
  for some classes of closed sets}, Math. Ann., 342 (2008), pp.~727--748.

\bibitem{AmbrosioFuscoPallara00}
{\sc L.~Ambrosio, N.~Fusco, and D.~Pallara}, {\em {Functions of bounded
  variation and free discontinuity problems}}, {Oxford Mathematical
  Monographs}, The Clarendon Press, Oxford University Press, New York, 2000.

\bibitem{cabre2000}
{\sc X.~Cabr\'{e}}, {\em Partial differential equations, geometry and
  stochastic control}, Butl. Soc. Catalana Mat., 15 (2000), pp.~7--27.

\bibitem{cabre2008}
\leavevmode\vrule height 2pt depth -1.6pt width 23pt, {\em Elliptic {PDE}'s in
  probability and geometry: symmetry and regularity of solutions}, Discrete
  Contin. Dyn. Syst., 20 (2008), pp.~425--457.

\bibitem{CRS2016}
{\sc X.~Cabr\'{e}, X.~Ros-Oton, and J.~Serra}, {\em Sharp isoperimetric
  inequalities via the {ABP} method}, J. Eur. Math. Soc. (JEMS), 18 (2016),
  pp.~2971--2998.

\bibitem{carpaspoz}
{\sc D.~Carazzato, G.~Pascale, and M.~Pozzetta}, {\em Quantitative
  isoperimetric inequalities in capillarity problems and cones in strong and
  barycentric forms}, 2025.
\newblock https://arxiv.org/abs/2507.07686.

\bibitem{ChGhoRi06}
{\sc J.~Choe, M.~Ghomi, and M.~Ritor\'{e}}, {\em Total positive curvature of
  hypersurfaces with convex boundary}, J. Differential Geom., 72 (2006),
  pp.~129--147.

\bibitem{ChGhoRi07}
\leavevmode\vrule height 2pt depth -1.6pt width 23pt, {\em The relative
  isoperimetric inequality outside convex domains in {${\bf R}^n$}}, Calc. Var.
  Partial Differential Equations, 29 (2007), pp.~421--429.

\bibitem{CGPRS}
{\sc E.~Cinti, F.~Glaudo, A.~Pratelli, X.~Ros-Oton, and J.~Serra}, {\em Sharp
  quantitative stability for isoperimetric inequalities with homogeneous
  weights}, Trans. Amer. Math. Soc., 375 (2022), pp.~1509--1550.

\bibitem{De-PhilippisMaggi15}
{\sc G.~De~Philippis and F.~Maggi}, {\em Regularity of free boundaries in
  anisotropic capillarity problems and the validity of {Y}oung's law}, Arch.
  Ration. Mech. Anal., 216 (2015), pp.~473--568.

\bibitem{dupuisishii}
{\sc P.~Dupuis and H.~Ishii}, {\em On oblique derivative problems for fully
  nonlinear second-order elliptic {PDE}s on domains with corners}, Hokkaido
  Math. J., 20 (1991), pp.~135--164.

\bibitem{Finnbook}
{\sc R.~Finn}, {\em Equilibrium capillary surfaces}, vol.~284 of Grundlehren
  der mathematischen Wissenschaften [Fundamental Principles of Mathematical
  Sciences], Springer-Verlag, New York, 1986.

\bibitem{FFLM22}
{\sc I.~Fonseca, N.~Fusco, G.~Leoni, and M.~Morini}, {\em Global and local
  energy minimizers for a nanowire growth model}, Ann. Inst. H. Poincar\'{e} C
  Anal. Non Lin\'{e}aire, 40 (2023), pp.~919--957.

\bibitem{FMMN}
{\sc N.~Fusco, F.~Maggi, M.~Morini, and M.~Novack}, {\em Rigidity and large
  volume residues in exterior isoperimetry for convex sets}, Adv. Math., 453
  (2024), pp.~Paper No. 109833, 27.

\bibitem{FM2023}
{\sc N.~Fusco and M.~Morini}, {\em Total positive curvature and the equality
  case in the relative isoperimetric inequality outside convex domains}, Calc.
  Var. Partial Differential Equations, 62 (2023), pp.~Paper No. 102, 32.

\bibitem{JWXZ}
{\sc X.~Jia, G.~Wang, C.~Xia, and X.~Zhang}, {\em Heintze-karcher inequality
  and capillary hypersurfaces in a wedge}, 2023.
\newblock ArXiv preprint 2209.13839.

\bibitem{kreutzschmidt}
{\sc L.~Kreutz and B.~Schmidt}, {\em A note on the winterbottom shape}, Proc.
  Roy. Soc. Edinburgh Sect. A, Published online (2024), pp.~1--14.

\bibitem{krummel}
{\sc B.~Krummel}, {\em Higher codimension relative isoperimetric inequality
  outside a convex set}, 2017.
\newblock ArXiv preprint 1710.04821.

\bibitem{LiuWangWeng}
{\sc L.~Liu, G.~Wang, and L.~Weng}, {\em The relative isoperimetric inequality
  for minimal submanifolds with free boundary in the {E}uclidean space}, J.
  Funct. Anal., 285 (2023), pp.~Paper No. 109945, 22.

\bibitem{lopez}
{\sc R.~L\'{o}pez}, {\em Capillary surfaces with free boundary in a wedge},
  Adv. Math., 262 (2014), pp.~476--483.

\bibitem{Maggi12}
{\sc F.~Maggi}, {\em {Sets of finite perimeter and geometric variational
  problems: an introduction to Geometric Measure Theory}}, vol.~135 of
  {Cambridge Studies in Advanced Mathematics}, Cambridge University Press,
  2012.

\bibitem{gasparetto}
{\sc L.~D. Masi, N.~Edelen, C.~Gasparetto, and C.~Li}, {\em Regularity of
  minimal surfaces with capillary boundary conditions}, Comm. Pure Appl. Math.,
  78 (2025), pp.~2436--2502.

\bibitem{Nittka}
{\sc R.~Nittka}, {\em Regularity of solutions of linear second order elliptic
  and parabolic boundary value problems on {L}ipschitz domains}, J.
  Differential Equations, 251 (2011), pp.~860--880.

\bibitem{paspoz}
{\sc G.~Pascale and M.~Pozzetta}, {\em Quantitative isoperimetric inequalities
  for classical capillarity problems}, Calc. Var. Partial Differential
  Equations, 63 (2024), pp.~Paper No. 225, 49.

\bibitem{Taylor77}
{\sc J.~E. Taylor}, {\em Boundary regularity for solutions to various
  capillarity and free boundary problems}, Comm. Partial Differential
  Equations, 2 (1977), pp.~323--357.

\end{thebibliography}
\end{document}